\documentclass[final,1p]{elsarticle}

\usepackage{amsmath,amssymb,exscale}
\usepackage{ams}
\usepackage{psfrag}
\usepackage{graphicx}
\usepackage[dvipsname]{xcolor}
\usepackage{enumerate}
\usepackage{tikz}

\newcommand{\corr}[1]{\textcolor{black}{#1}}
\newcommand{\cor}[1]{\textcolor{black}{#1}}
\newcommand{\corg}[1]{\textcolor{black}{#1}}
\newcommand{\corm}[1]{\textcolor{black}{#1}}

\newcommand{\corred}[1]{\textcolor{black}{#1}}

\newcommand{\mcI}{{\mathcal I}}

\newcommand{\mcP}{{\mathcal P}}

\newcommand{\mcS}{\mathcal{S}}

\newcommand{\Real}{\mathop{\text{\rm Re}}}
\newcommand{\Imag}{\mathop{\text{\rm Im}}}

\newcommand{\jj}{\mathbf{j}}
\newcommand{\pp}{\mathbf{p}}

\newcommand{\bq}{{\bf q}}
\newcommand{\bw}{{\bf w}}
\newcommand{\by}{{\bf y}}
\newcommand{\bz}{{\bf z}}

\newcommand{\ii}{\mathbf{i}}

\newcommand{\mm}{\mathbf{m}}

\newcommand{\oone}{\boldsymbol{1}}

\newcommand{\E}{{\mathbb E}}

\newcommand{\rset}{{\mathbb R}}

\newcommand{\Nset}{{\mathbb N}}

\newcommand{\Pol}{\mathbb{P}}
\newcommand{\eset}[1]{{\mathbb E} \left[ #1 \right] }

\newcommand{\Grad}{\nabla}

\newcommand{\var}{\text{\rm var}}

\newcommand{\lv}{w}

\newtheorem{assumption}{Assumption}

\newtheorem{remark}{Remark}
\newtheorem{problem}{Problem}
\newtheorem{prop}{Proposition}

\newtheorem{thm}{Theorem}
\newtheorem{lem}[thm]{Lemma}
\newtheorem{corl}[thm]{Corollary}
\newdefinition{rmk}{Remark}
\newproof{pf}{Proof}
\newproof{pot}{Proof of Theorem \ref{thm2}}

\DeclareMathOperator*{\esssup}{ess\,sup}
\DeclareMathOperator*{\essinf}{ess\,inf}
\biboptions{sort&compress}

\journal{Computers and Mathematics with applications}

\begin{document}

\begin{frontmatter}

\title{\corr{Analytic regularity and collocation approximation for
    elliptic PDEs with random domain deformations}}

\author[jc]{Julio E. Castrill\'{o}n-Cand\'{a}s\corref{cor1}}
\ead{uvel@alum.mit.edu}

\author[fn]{Fabio Nobile}
\ead{fabio.nobile@epfl.ch}

\author[rt,jc]{Ra\'{u}l F. Tempone}
\ead{raul.tempone@kaust.edu.sa}

\address[jc]{SRI Center for Uncertainty Quantification, 4700 King Abdullah
University of Science and Technology, Thuwal 23955-6900, Saudi Arabia.}

\address[fn]{CSQI-MATHICSE, Ecole Politechnique Federale Lausanne,
Station 8, CH1015, Lausanne, Switzerland}

\address[rt]{Applied Mathematics and Computational Science, 4700,
  King Abdullah University of Science and Technology, Thuwal,
  23955-6900, Saudi Arabia}

\cortext[cor1]{corresponding author}

\begin{abstract}
In this work we consider the problem of approximating the statistics
of a given Quantity of Interest (QoI) that depends on the solution of
a linear elliptic PDE defined over a random domain parameterized by
$N$ random variables. The elliptic problem is remapped on to a
corresponding PDE with a fixed deterministic domain. We show that the
solution can be analytically extended to a well defined region in
$\C^{N}$ with respect to the random variables. A sparse grid
stochastic collocation method is then used to compute the mean and
standard deviation of the QoI. Finally, convergence rates for the mean
and variance of the QoI are derived and compared to those obtained in
numerical experiments.
\end{abstract}

\begin{keyword} Uncertainty Quantification,
Stochastic Collocation, Stochastic PDEs, Finite Elements, Complex Analysis,
Smolyak Sparse Grids
\end{keyword}

\end{frontmatter}



\section{Introduction}
In many physical processes the practicing engineer or scientist
encounters the problem of optimal design under uncertainty of the
underlying domain. For example, in graphine sheet nano fabrication the
exact geometries of the designed patterns (e.g. nano pores) are not
easy to control due to uncertainties.  If there is no quantitative
understanding in the involved domain uncertainty such a design may be
carried out by trial and error. However, in order to accelerate the
design cycle, it is essential to quantify the influence of this
uncertainty on Quantities of Interest, \corr{for example,} the sheet
stress of the graphene sheet. Other examples include lithographic
process introduced in semi-conductor design \cite{Zhenhai2005}.

Collocation and perturbation approaches have been suggested in the
past as an approach to quantify the statistics of the QoI with random
domains \cite{Zhenhai2005, Chauviere2006, Fransos2008, 
  Tartakovsky2006, Harbrecht2008}.  The collocation approaches
  proposed in \cite{Chauviere2006, Fransos2008, Tartakovsky2006} work
  well for large amplitude domain perturbations although suffer from
  the curse of dimensionality. Moreover, these works lack error
estimates of the QoI with respect to the number of sparse grid points.
On the other hand, the perturbations approaches introduced in
\cite{Harbrecht2008, Zhenhai2005} are efficient for small domains
\corr{perturbation}.

In this paper we give a rigorous convergence analysis of the
collocation approach based on isotropic Smolyak grids. This consists
of an analysis of the regularity of the solution with respect to the
parameters describing the domain perturbation. In this respect we show
that the solution can be analytically extended to a well defined
region in $\C^{N}$ with respect to the random variables. Moreover, we
derive error estimates both in the ``energy norm'' as well as on
functionals of the solution (Quantity of Interest) for Clenshaw Curtis
abscissas that can be easily generalized to a larger class of sparse
grids.

The outline of the paper is the following: In Section \ref{setup} we
set up the mathematical problem and reformulate the random domain
elliptic PDE problem onto a deterministic domain with random matrix
coefficients. We assume that the random boundary is parameterized by
$N$ random variables.  In Section \ref{analyticity} we show that the
solution can be analytically extended into a well defined region in
$\C^{N}$. Theorem \ref{analyticity:theorem1} is the main result of
this paper.  In Section \ref{stochasticcollocation} we setup the
stochastic collocation problem and summarize several known
\corr{sparse grid} approaches that are used to approximate the mean
and variance of the QoI.  In Section \ref{erroranalysis} we assume
that the random domain is truncated to $N_s \leq N$ random
variables. We derive error estimates for the mean and variance of the
QoI with respect to the finite element, sparse grid and truncation
approximations. Finally, in section \ref{numericalresults} numerical
examples are presented.

\section{Setup and problem formulation}
\label{setup}

Let $\Omega$ be the set of outcomes from the complete probability
space $(\Omega, {\cal F}, \mathbb{P})$, where ${\cal F}$ is a sigma
algebra of events and $\mathbb{P}$ is a probability measure.  Define
$L^{q}_{P}(\Omega)$, $q \in [1, \infty]$, as the space of random
variables such that
\[
L^{q}_{P}(\Omega):= \corm{\Bigg\{ v \,\,|\,\, \int_{\Omega}
|v(\omega)|^{q}\,d\mathbb{P} < \infty \Bigg\} } \,\,\mbox{and}\,\,\,
L^{\infty}_{P}(\Omega) := \{v \,\,|\,\, \esssup_{ \omega \in \Omega}
|v(\omega)| < \infty \},
\]
\noindent where $v:\Omega \rightarrow \mathbb{R}$ be a measurable
random variable.

Suppose $D(\omega) \subset \mathbb{R}^{d}$ is an open bounded domain
with Lipschitz boundary $\partial D(\omega)$ parameterized with
respect to a stochastic parameter $\omega \in \Omega$. The strong form
of the problem we consider in this work is: \cor{given sufficiently
  smooth regularity on $f(\cdot,\omega), a(\cdot,\omega): D(\omega)
  \rightarrow \R^{d}$},
find $\corr{u(\cdot,\omega)}: D(\omega) \rightarrow \R$ such that
almost surely
\[
\begin{split}
  -\nabla \cdot ( a(x, \omega) \nabla u(x,\omega) ) &= f(x,\omega),\,\,\,
\corm{x \in D(\omega),} \\
u &=0\hspace{11mm}\mbox{on \corm{$\partial D(\omega)$.}}
\end{split}
\] 
Now, assume the diffusion coefficient satisfies the following
assumption.
\begin{assumption} There exist constants $a_{min}$ and $a_{max}$ such
  that
\[
0 < a_{min} \leq a(x, \omega) \leq a_{max} < \infty
\,\,\,\mbox{for a.e. $x
    \in D(\omega)$, $\omega \in \Omega$},
\]
\noindent where
\[
a_{min} := \essinf_{x \in D(\omega), \omega \in \Omega } a(x, \omega)
\,\,\,\,\mbox{and}\,\,\,\, a_{max} := \esssup_{x \in D(\omega),
  \omega \in \Omega} a(x, \omega).
\]
\label{setup:Assumption1}
\end{assumption}
We now state the weak formulation as:
\begin{problem}
Find $u(\cdot,\omega) \in H^{1}_{0}(D(\omega))$ s.t.
\begin{equation}
\int_{D(\omega)} a(x, \omega) 
\nabla u(x, \omega)  \cdot \nabla v(x) \,\, dx =
\int_{D(\omega)} f(x, \omega) v(x) \,\, dx \,\,\,\,\, \forall v \in
H^{1}_{0}(D(\omega))\,\,\,\mbox{$a.s.$ in $\Omega$},
\end{equation}
\noindent where $f(\cdot,\omega) \in L^{2}(D(\omega))$ for
a.e. $\omega \in \Omega$.
\label{setup:Prob1}
\end{problem}
Under Assumption \ref{setup:Assumption1} the weak formulation has a
unique solution up to a zero-measure set in $\Omega$.

\subsection{Reformulation onto a fixed Domain}

Now, assume that given any $\omega \in \Omega$ the domain $D(\omega)$
can be mapped to a \corr{open and bounded} reference domain $U \subset
\mathbb{R}^{d}$ with Lipschitz boundary through a random map
\corr{$F(\omega):U \rightarrow D(\omega)$}, where we assume that
\corr{$F(\omega)$} is one-to-one and the determinant of the
\corr{Jacobian $|\partial F(\cdot, \omega)| \in L^{\infty}(U)$ almost
  surely.}  Furthermore, we assume that $|\partial F|$ is uniformly
greater than zero almost surely.  We will, however, make the following
equivalent assumption.

\begin{assumption} \corr{Suppose that the  map $F(\omega):U \rightarrow
  D(\omega)$ is one-to-one a.s. and that there exist constants
  $\F_{min}$ and $\F_{max}$ such that}
\[
0<\F_{min} \leq \sigma_{min} (\partial F(\omega)) 
\,\,\mbox{and}\,\, \sigma_{max} (\partial F(\omega)) \leq
\F_{max} < \infty
\]
\label{setup:Assumption2}
\noindent almost everywhere in $U$ and almost surely in $\Omega$. We
have denoted by $\sigma_{min} (\partial F(\omega))$ (and $\sigma_{max}
(\partial F(\omega))$) the minimum (respectively maximum) singular
value of the Jacobian $\partial F(\omega)$. 
\end{assumption}
\cor{In the rest of the paper we shall drop repeating a.s. in $\Omega$
  and a.e.  in $U$ unless disambiguation is needed.  Moreover,
    for any random function $v(\omega):D(\omega)\rightarrow \R$, we
    denote by $\hat{v} = v \circ F$ the random function
    $\hat{v}(\cdot, \omega) = v(F(\cdot,\omega),\omega):U
    \rightarrow \R$.}

\smallskip

Problem \ref{setup:Prob1} can be reformulated with respect to the
fixed reference domain $U$. From the chain rule we have that for any
$v \in C^{1}(D(\omega))$
\begin{equation}
\nabla v = \partial F^{-T} \nabla (v \circ F).
\label{setup:equivalence}
\end{equation}
\begin{remark}
Note that we refer to $\nabla v:D(\omega) \rightarrow \R^d$ as
  the gradient of $v$ on $D(\omega)$ and $\nabla \hat v:U \rightarrow
  \R^d$, with $\hat v = v \circ F$, as the gradient of \corred{$\hat v$} on
  $U$. Therefore the notation $\nabla (v \circ F):U \rightarrow \R^d$
  refers to the gradient of $\hat v$ on $U$.
\end{remark}

\corr{Now let us recall the chain rule for Sobolev spaces (see Theorem
  3.35 in \cite{Adams1975}): Let $U$, $D \subset \R^n$ and $T: U
  \rightarrow D$ be invertible. Suppose that $T$ and $T^{-1}$ have
  continuous, bounded derivatives of order less or equal to $j$. Then
  if $k \in W^{j,p}(D)$, $p \geq 1$, we have $v = k \circ T \in
  W^{j,p}(U)$ and the derivatives of $v$ are given by the chain rule.}

Thus by the chain rule on Sobolev spaces for any $v \in
H^{1}(D(\omega))$ we have that $\nabla v = \partial F^{-T} \nabla (v
\circ F)$, where $v \circ F \in H^{1}(U)$.  By a change of variables,
the weak form can now be posed as:
\begin{problem}
Find  $\hat u(\cdot,\omega) \in H^{1}_{0}(U)$ s.t.
\begin{equation}
B(\omega ;\hat u(\cdot,\omega),v) = l(\omega; v) ,\,\,\,\forall
\corm{v \in H^{1}_{0}(U),}
\end{equation}
\noindent where for any $v,s \in H^1_{0}(U)$
\[
\begin{split}
B(\omega; s,v)
&:= 
\int_{U} \corg{(a \circ F)(\cdot,\omega)(\nabla s)}
^{T} C^{-1}(\cdot,\omega) \nabla v | \partial
F(\cdot,\omega)|, \\
l(\omega; v)
&:=
\int_{U} \cor{(f \circ F)(\cdot,\omega)}  v \,\, | \partial F(\cdot,\omega)  |,
\end{split}
\]
\label{setup:Prob2}
$(f \circ F)(\cdot,\omega) \in L^{2}(U)$ and $C(\cdot,\omega) =
  \partial F^{T} (\cdot,\omega) \partial F(\cdot,\omega) $ almost
  surely.  We now recover $u(\cdot,\omega):D(\omega)\rightarrow
  H^{1}_{0}(\omega)$ as $u = \hat{u} \circ F^{-1}$.
\end{problem}
Note that under this notation $\hat{u}(\cdot,\omega)$ can be
  written as $u(\cdot,\omega) \circ F(\cdot,\omega)$ or shortly as $u
  \circ F$, which is the notation used in the rest of the paper. Note
  that we can also use the notation $(u \circ F)(\cdot,\omega)$.


The following lemmas give the conditions under which Problem
\ref{setup:Prob2} is well posed.

\begin{lem}
\corr{Under Assumptions \ref{setup:Assumption2} we have that
\begin{enumerate}[i)]
\item $L^{2}(D(\omega))$ and $L^{2}(U)$ are isomorphic.
\item $H^{1}(D(\omega))$ and $H^{1}(U)$ are isomorphic.
\end{enumerate}
Moreover, $\forall v \in H^{1}(D(\omega))$ 
\begin{equation}
\corr{\| \nabla v \|_{L^{2}(D(\omega))} 
\leq \F^{d/2}_{max}\F^{-1}_{min}
\| \nabla (v \circ F) \|_{L^{2}(U)}}.
\label{setup:eqn0}
\end{equation}
}
\label{setup:lemma0}
\end{lem}
\begin{pf}
\corr{ $i)$ is immediate.  Now, from \ref{setup:equivalence} and the
  chain rule on Sobolev spaces we obtain that $\forall v \in
  H^{1}_{0}(D(\omega))$ the inequality \eqref{setup:eqn0} is
  satisfied. We can similarly obtain a bound for the converse.  It
  follows that $H^{1}(D(\omega))$ and $H^{1}(U)$ are isomorphic.}
\end{pf}

\begin{lem} Given that Assumptions \ref{setup:Assumption1} and
  \ref{setup:Assumption2} are satisfied then there exists a.s. a
  unique solution to Problem \ref{setup:Prob2}, which coincides with
  the solution to Problem \ref{setup:Prob1}, and 
\[
\| \nabla u \|_{L^{2}(D(\omega))} 
 \leq 
\frac{\F^{3d/2+2}_{max}}{a_{min} \F^{d+1}_{min}} \| f \circ F
  \|_{L^{2}(U)} C_{P}(U)
\]
\noindent where $C_{P}(U)$ is the Poincar\'{e} constant of the
reference domain $U$.
  \label{setup:lemma1}
\end{lem}
\begin{pf} 
  From Assumption \ref{setup:Assumption2} we have that
\[
|\partial F| = \sqrt{ |C|} 
            = \sqrt{\Pi_{i=1}^{d} \lambda(C) } 
            = \Pi_{i=1}^{d} \sigma_{i}( \partial F).
\]
\noindent therefore $\F_{min}^{d} \leq |\partial F| \leq
\F_{max}^{d}$.  Furthermore, from Assumption \ref{setup:Assumption2}
we have that 
\[
\lambda_{min} ((a \circ F) C^{-1} |\partial F| ) 
\geq a_{min}
\F^{d}_{min} \lambda_{min}(C^{-1}) = a_{min} \F^{d}_{min} \F^{-2}_{max} 
> 0,
\]
and
 \[
 \lambda_{max} ( ( a \circ F ) C^{-1} |\partial F| ) \leq a_{max}
 \F^{d}_{max} \lambda_{max}(C^{-1}) = a_{max} \F^{d}_{max}\F^{-2}_{min} 
 < \infty.
 \]
Thus Problem \ref{setup:Prob2} is uniformly continuous and coercive.
\corr{Furthermore, since $f \circ F \in L^{2}(U)$ then from the
  Lax-Milgram theorem there exists a.s. a unique solution}.  The
equivalence between Problems \ref{setup:Prob1} and \ref{setup:Prob2}
is an immediate consequence of the chain rule and the isomorphism
between $H^{1}_{0}(U)$ and $H^{1}_{0}(D(\omega))$ (Lemma
\ref{setup:lemma0}).

\noindent From the Cauchy-Schwartz inequality we obtain
\[
\begin{split}
\lambda_{min} ((a \circ F) C^{-1} |\partial F| ) 
 \| \nabla (u \circ F) \|^{2}_{L^{2}(U)} 
&\leq 
|B(\omega;u \circ F, u \circ F)| = |l(\omega;u \circ F)| \\
&\leq
\int_{U} |f \circ F | |u\circ F| |\partial F|\\
&\leq \|f \circ F\|_{L^{2}(U)}  \|u \circ F\|_{L^{2}(U)}  \F^{d}_{max}.
\end{split}
\]
From the Poincar\'{e} inequality ($\|u \circ F\|_{L^{2}(U)} 
\leq C_{P}(U) \| \nabla (u \circ F)\|_{L^{2}(U)}$) we obtain
\[
\| \nabla (u \circ F) \|_{L^{2}(U)}
 \leq \frac{\| f \circ F
  \|_{L^{2}(U)} C_{P}(U) \F^{d}_{max}  
}{a_{min} \F^{d}_{min} 
\F^{-2}_{max}}.
\]
\corr{Now, from equation \eqref{setup:eqn0} it follows that $\forall v
  \in H^{1}_{0}(D(\omega))$}
\[
\| \nabla u \|_{L^{2}(D(\omega))}
 \leq \frac{\F^{3/2d+2}_{max}}{a_{min} \F^{d+1}_{min}} \| f \circ F
  \|_{L^{2}(U)} C_{P}(U).
\]
\end{pf}

\begin{remark}
\label{setup:Inhomogenous}

For many practical applications the non-zero Dirichlet boundary value
problem is more interesting. We can easily extend the stochastic
domain problem to non-zero Dirichlet boundary conditions.

Suppose we have the following boundary value problem: Given
  $f(\cdot,\omega), a(\cdot,\omega): D(\omega) \rightarrow \R^{d}$ and
  $g(\cdot,\omega):\partial D(\omega) \rightarrow \R^{d}$ find
  $u(\cdot,\omega): D(\omega) \rightarrow \R^{d}$ such that almost
  surely
\[
\begin{split}
  -\nabla \cdot ( a(x, \omega) \nabla u(x,\omega) ) &= f(x,\omega),\,\,\,
\corm{x \in D(\omega),} \\
u &= g\hspace{11mm}\mbox{on $\partial D(\omega)$.}
\end{split}
\]
Since $U$ is bounded and Lipschitz there exists a bounded linear
operator $T:H^{1/2}(\partial U) \rightarrow H^{1}(U)$ such that
$\forall \tilde g \in H^{1/2}(\partial U)$ we have that $\tilde {\bf
  w} := T \tilde g \in H^{1}(U)$ satisfies $\hat \bw|_{U} = \tilde g$
almost surely. The weak formulation can now be posed as
(\cite{evans1998} chapter 6, p297):
\begin{problem}
Given that $f \circ F \in L^{2}(U)$ find \cor{$\hat{u}(\cdot, \omega)
  \in H^{1}_{0}(U)$} s.t.
\[
  \cor{B(\omega ;\hat{u},v) = 
\tilde{l}(\omega; v) ,\,\,\,\forall v \in H^{1}_{0}(U) }
\]
\corred{almost surely, where $\tilde{l}(\omega; v):=\int_{U} (f \circ
  F)(\cdot,\omega) | \partial F(\cdot,\omega) |v - L(\hat
  \bw(\cdot,\omega),v)$, $\hat g := g \circ F$, $\hat \bw := T(\hat
  g)$,
\[
L( \hat \bw(\cdot,\omega) ,v):=
   \int_U (a \circ F)(\cdot,\omega) (\nabla (\hat \bw
(\cdot,\omega))^{T} C^{-1}(\cdot,\omega) | \partial F(\cdot,
  \omega) | \nabla v, 
\]
and $\hat \bw(\cdot,\omega)
  |_{\partial U} = \hat g(\cdot,\omega)$.  This homogeneous
  boundary value problem can be remapped to $D(\omega)$ as
  $\tilde{u}(\cdot,\omega) := (\hat{u} \circ F^{-1})(\cdot, \omega)$,
  thus we can rewrite $\hat{u}(\cdot,\omega) = (\tilde{u} \circ
  F)(\cdot,\omega)$.}
\label{setup:Prob3}
\end{problem}
The solution $u(\cdot, \omega) \in H^{1}(D(\omega))$ for the non-zero
Dirichlet boundary value problem is obtained as $u(\cdot,\omega) =
\tilde{u} (\cdot,\omega) + \corred{\hat \bw \circ F^{-1}(\cdot,\omega)
}$.
\end{remark}

\subsubsection{Quantity of Interest and the Adjoint problem}
\label{setup:QoI}
In practice we are interested in computing the statistics of a
Quantity of Interest (QoI) over the stochastic domain or a subdomain
of it. We consider QoI of the form
\begin{equation}
\cor{Q(u) := \int_{\tilde{D}} q(x) u(x,\omega)\,\,dx}
\label{setup:qoi}
\end{equation}
\cor{with $q \in L^{2}(\tilde D)$ over the region $\tilde{D} \subset
  D(\omega)$ for any $\omega \in \Omega$.  Moreover, we assume that
  $\exists \delta > 0$ such that $dist(\tilde{D}, D(\omega)) < \delta$
  $\forall \omega \in \Omega$ and $F \left. \right|_{\tilde{D}} = I$
  on $\tilde{D}$ (i.e. there is no deformation of the domain on
  $\tilde{D}$).}

In this paper we restrict our attention to the computation of the mean
$\mathbb{E}[Q]$ and variance $Var[Q]:= \mathbb{E}[Q^{2}] -
\mathbb{E}[Q]^{2}$ given that the domain deformation is parameterized
by a stochastic random vector.

We first assume that $Q:H^{1}_{0}(U) \rightarrow \R$ is
a bounded linear functional. The influence function can be computed
as:
\begin{problem}
Find $\varphi \in H^{1}_{0}(U)$ such that $\forall v \in H^{1}_{0}(U)$
\begin{equation}
B(\omega; v, \varphi) = Q(v)
\label{setup:dual2}
\end{equation}
a.s. in $\Omega$.
\label{setup:Prob4}
\end{problem}

\corred{We can now pick a particular $T$ such that $\hat
  \bw=T( \hat g)$ vanishes inside $\tilde{D}$.
Therefore, we have that}
\[
Q(u) = Q(\tilde{u} + \hat \bw \circ F^{-1}) = Q(\tilde{u}),
\]
\cor{and thus}
\[
\cor{Q(\tilde u) = \int_{\tilde{D}} q(x) \tilde u(x,\omega)\,\,dx =
\int_{\tilde{D}} q \circ F (x) \tilde u \circ F(x,\omega)\,\,dx
= Q(\tilde u \circ F) = B(\omega; \tilde{u} \circ F,\varphi).}
\]


\subsection{Domain Parameterization}
\label{setup:domainparametrization}

Let $Y:=[Y_1, \dots, Y_{N}]$ be a $N$ valued random vector measurable
in $(\Omega, {\cal F}, \mathbb{P})$ taking values on
$\Gamma:=\Gamma_{1} \times \dots \times \Gamma_{N} \subset
\mathbb{R}^{N}$ and ${\cal B}(\Gamma)$ be the Borel $\sigma-$algebra.

Define the induced measure $\mu_{Y}$ on $(\Gamma,{\cal B}(\Gamma))$ as
$\mu_{Y} : = \mathbb{P}(Y^{-1}(A))$ for all $A \in {\cal
  B}(\Gamma)$. Assuming that the induced measure is absolutely
  continuous with respect to the Lebesgue measure defined on $\Gamma$,
  then there exists a density function $\rho({\bf y}): \Gamma
  \rightarrow [0, +\infty)$ such that for any event $A \in {\cal
      B}(\Gamma)$
\[
\mathbb{P}(Y \in A) := \mathbb{P}(Y^{-1}(A)) = \corm{\int_{A} \rho( {\bf y}
)\,d {\bf y}.}
\]
\noindent Now, for any measurable function $Y \in L^{1}_{P}(\Gamma)$
we let the expected value be defined as
\[
\mathbb{E}[Y] = \int_{\Gamma} {\bf y} \rho( {\bf y} )\, d
{\bf y}.
\]
\corr{Further, we define the following spaces:}
\[
\corr{L^{q}_{\rho}(\Gamma) := \left\{ v \,\,|\,\, \int_{\Omega}
|v(\by)|^{q}\, \rho(\by) d\by < \infty \right\} \,\,\mbox{and}\,\,\,
L^{\infty}_{\rho}(\Gamma) := \left\{v \,\,|\,\, \esssup_{ \by \in \Gamma}
|v(\by)| < \infty \right\}.}
\]
The mapping $F(\cdot, \omega):U \rightarrow D(\omega)$ can be
parameterized in many forms.  In this paper we restrict our attention
to the following class of mappings:

\begin{assumption} 
\corr{The map $F(\omega):U \rightarrow D(\omega)$ has the form}
\[
F(x, \omega) := x + e(x,\omega)\hat{v}(x)
\]
\corr{a.s. in $\Omega$, with $\hat{v}:U \rightarrow \R^{d}$, $\hat v :=
[\hat v_1, \dots, \hat v_d]^T$, $\hat{v}_{i} \in C^{1}(U)$ for $i = 1,
\dots, d$, and $e(\cdot,\omega): U \rightarrow
D(\omega)$. Assume that the map $F(\omega): U \rightarrow
D(\omega)$ is one-to-one almost surely.}
\label{setup:Assumption3}
\end{assumption}

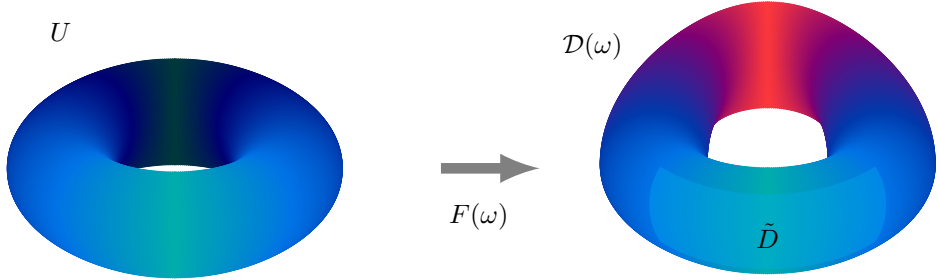
\begin{figure}
\begin{center}
\begin{tikzpicture}[scale=1] 
    \foreach \x in {90,...,-90} { 
    \pgfmathsetmacro\elrad{20*max(cos(\x),.1)}
    \pgfmathsetmacro\ltint{.9*abs(\x-45)/180}
    \pgfmathsetmacro\rtint{.9*(1-abs(\x+45)/180)}
    \definecolor{currentcolor}{rgb}{0, \ltint, \rtint}
    \draw[color=currentcolor,fill=currentcolor] 
        (xyz polar cs:angle=\x,y radius=0.75,x radius=1.5) 
        ellipse (\elrad pt and 20pt);
    \definecolor{currentcolor}{rgb}{0, \ltint, \rtint}
    \draw[color=currentcolor,fill=currentcolor] 
        (xyz polar cs:angle=180-\x,radius=0.75,x radius=1.5) 
        ellipse (\elrad pt and 20pt);
    }
\coordinate (O) at (3.5,0);
\coordinate (P) at (4.8,0);
\draw[->, >=latex, gray, line width=4 pt] (O) -- (P);
\node[above=16pt,scale=1] at (4,-1.5) {$F(\omega)$};
\node[above=16pt,scale=1] at (-1.5,1) {$U$};
\end{tikzpicture}
\begin{tikzpicture}[scale=1] 
    \foreach \x in {90,89,...,0} { 
    \pgfmathsetmacro\elrad{20*max(cos(\x),.1)}
    \pgfmathsetmacro\ltint{.9*abs(\x-45)/180}
    \pgfmathsetmacro\rtint{.9*(1-abs(\x+45)/180)}
    \pgfmathsetmacro\ztint{2*(abs(\x)/180)}
    \definecolor{currentcolor}{rgb}{\ztint,\ltint,\rtint}
    \draw[color=currentcolor,fill=currentcolor] 
        (xyz polar cs:angle=\x,y radius=1.45,x radius=1.5) 
        ellipse (\elrad pt and 20pt);
    \definecolor{currentcolor}{rgb}{\ztint,\ltint,\rtint}
    \draw[color=currentcolor,fill=currentcolor] 
        (xyz polar cs:angle=180-\x,radius=1.45,x radius=1.5) 
        ellipse (\elrad pt and 20pt);
    }

    \foreach \x in {0,...,-90} { 
    \pgfmathsetmacro\elrad{20*max(cos(\x),.1)}
    \pgfmathsetmacro\ltint{.9*abs(\x-45)/180}
    \pgfmathsetmacro\rtint{.9*(1-abs(\x+45)/180)}
    \definecolor{currentcolor}{rgb}{0, \ltint, \rtint}
    \draw[color=currentcolor,fill=currentcolor] 
        (xyz polar cs:angle=\x,y radius=.75,x radius=1.5) 
        ellipse (\elrad pt and 20pt);
    \definecolor{currentcolor}{rgb}{0, \ltint, \rtint}
    \draw[color=currentcolor,fill=currentcolor] 
        (xyz polar cs:angle=180-\x,radius=.75,x radius=1.5) 
        ellipse (\elrad pt and 20pt);
    }
    \pgfmathsetmacro\elrad{cos(-135)}
    \pgfmathsetmacro\xrad{1.5cm-20pt*\elrad}
    \pgfmathsetmacro\yrad{.75cm-20pt*sin(-135)}
    \path[fill=cyan, fill opacity=.35] 
        (xyz polar cs:angle=-135,radius=.75,x radius=1.5) 
        ++(20pt*\elrad,0) arc (0:45:20*\elrad pt and 20pt) 
        arc (-135:-45:\xrad pt and \yrad pt) 
        arc (45:-45:-20*\elrad pt and 20pt) 
        arc (-45:-135:\xrad pt and \yrad pt) 
        arc (-45:0:20*\elrad pt and 20pt);

    \node[above=16pt,scale=1] at (0,-1.8) {$\corr{\tilde{D}}$};
    \node[above=16pt,scale=1] at (-2.3,0.7) {${\cal D}(\omega)$};
\end{tikzpicture}
\end{center}
\caption{\corr{Cartoon example of stochastic domain realization from a
    reference domain. Note that the non stochastic area $\corr{\tilde{D}}$ is
    contained in the interior of $D(\omega)$. Note that this figure is
    modified from the TikZ tex code from {\it Smooth map of manifolds
      and smooth spaces} by Andrew Stacey.}}
\label{setup:fig1}
\end{figure}

\cor{We next assume the stochastic perturbation $e(x,\omega)$ with the
following parameterization:}
\begin{equation}
\corred{e(x, \omega) := \sum_{l=1}^{N} \sqrt{\mu_{l}}
b_{l}(x)Y_{l}(\omega).}
\label{setup:expansion}
\end{equation}
\cor{Denote $Y : = [Y_{1}, \dots, Y_{N}]$, and for $n = 1, \dots, N$
  let $\Gamma_{n}\equiv Y_{n}(\Omega)$ and $\mathbb{E}[Y_{n}]=0$.}
Furthermore denote $\Gamma:= \prod_{n=1}^{N} \Gamma_{n}$, and
$\rho({\bf y}):\Gamma \rightarrow \mathbb{R}_{+}$ as the joint
probability density of $Y$.  In addition, we make the following
assumptions:
\begin{assumption}~
\corm{\begin{enumerate}[1)]
\item 
$n = 1, \dots, N$, $\Gamma_{n} \equiv [-1,1]$. 
\item $b_{1},\dots, b_{N} \in W^{1,\infty}(U)$.
\item $\| b_l \|_{L^{\infty}(U)} = 1$ for $l = 1,2,\dots N$.
\item  $\mu_{l}$ are monotonically decreasing for $l = 1,2,\dots N$.
\end{enumerate}}
\label{setup:Assumption4}
\end{assumption}
\corr{Note that the parameterization of equation \eqref{setup:expansion} may
arise from a truncated Karhunen Loeve (KL) expansion of mean square
random fields. However, in general KL expansion eigenfunctions $b_l$
belong to $L^{2}(U)$, but can be extended to $L^{\infty}(U)$ or even
higher regularity if the covariance function satisfies certain
regularity properties (see \cite{Frauenfelder2005} details).}

\section{Analyticity}
\label{analyticity}

\corm{The analytic extension of the solution of Problem \ref{setup:Prob3},
with respect to the images of the stochastic variables, provides us a
form to bound the approximation error of the collocation scheme.}  In
this section we analyze the analyticity of the solution assuming a
stochastic domain but with deterministic $f$.


From the stochastic model formulated in Section \ref{setup} the
Jacobian $\partial F$ is written as
\begin{equation}
\partial F(x,\omega) = I + 
\sum_{l=1}^{N} B_{l}(x) \sqrt{\mu_{l}}  Y_{l}(\omega) 
\label{analyticity:eqn1}
\end{equation}
with
\[
B_{l}(x) := b_{l}(x) \partial \hat{v}(x) + 
\left[
 \begin{array}{cccc}
 \frac{\partial b_{l}(x)}{\partial x_{1}} \hat{v}_{1}(x)
 & \frac{\partial b_{l}(x)}{\partial x_{2}} \hat{v}_{1}(x) & \dots &
 \frac{\partial b_{l}(x)}{\partial x_{d}} \hat{v}_{1}(x) \\
 \frac{\partial b_{l}(x)}{\partial x_{1}} \hat{v}_{2}(x)
 & \frac{\partial b_{l}(x)}{\partial x_{2}} \hat{v}_{2}(x) & \dots &
 \frac{\partial b_{l}(x)}{\partial x_{d}} \hat{v}_{2}(x) \\
 \vdots & \vdots & \vdots & \vdots \\ 
 \frac{\partial b_{l}(x)}{\partial x_{1}} \hat{v}_{d}(x)
 & \frac{\partial b_{l}(x)}{\partial x_{2}} \hat{v}_{d}(x) & \dots &
 \frac{\partial b_{l}(x)}{\partial x_{d}} \hat{v}_{d}(x) \\
 \end{array}
 \right],
\]
\noindent where $\partial \hat v$ is the Jacobian of $\hat v(x)$.

\begin{assumption}~
\begin{enumerate}[a)]
\item $a \circ F$ and $\hat g$ are only a function of $x
  \in U$ and independent of $\omega \in \Omega$.

\item There exists $0 < \tilde{\delta} < 1$ such that $\sum_{l=1}^{N}
   \| B_l(x) \|_{2} \sqrt{\mu_{l}} \leq 1 -
   \tilde{\delta}$, $\forall x \in U$.
\end{enumerate}
\label{analyticity:assumption1}
\end{assumption}

\begin{remark}
   \corm{Assumption \ref{analyticity:assumption1} a)} restricts
     $a(x,\omega)$ (and $\hat{g}(x,\omega)$) to be a constant
     along the direction $\hat v(x)$.  This assumption simplifies the
     presentation of this section. However, \corm{Assumption
     \ref{analyticity:assumption1} a)} is still useful for many
     practical problems. For example, in layered materials the width
     or geometry of each layer is stochastic, but the diffusion
     coefficient for each layer is known. This example can be found in
     semi-conductor design \cite{Zhenhai2005}. Moreover, this
     assumption allows the diffusion coefficient to be non constant
     along the non stochastic directions.
\end{remark}

\corr{We now extend the mapping $\partial F(x,{\bf y}) = I + R(x,{\bf y})$},
with $R(x,{\bf y}) := \sum_{l=1}^{N} \sqrt{\mu_{l}}
B_{l}(x)y_{l}$, to the complex plane.  First, for any $0 < \beta <
\tilde{\delta}$ define the following region in $\C^{N}$:
\begin{equation}
\corr{
\Theta_{\beta} : = \left\{ {\bf z} \in \mathbb{C}^{N};\,
{\bf z} = {\bf y} + \bw,\,{\bf y} \in [-1,1]^{N},\,
 \sum_{l=1}^{N}  \sup_{x \in U}  \| B_l(x) \|_{2} 
\sqrt{\mu_{l}} |w_{l}| \leq \beta
\right\},
}
\label{analyticity:region}
\end{equation}
\corr{where $\bw:=[w_1, \dots, w_N] \in \mathbb{C}^{N}$.}
 \begin{assumption}
 \corm{Assume that $f: \R^d \rightarrow \R$ can be analytically
   extended in $\C^{d}$. Furthermore assume that the analytic
   extension $\Real (f \circ F)(\cdot,\bz), \Imag (f \circ
   F)(\cdot,\bz) \in L^{2}(U)$ $\forall \bz \in \Theta_{\beta}$.}
\label{analyticity:assumption2}
\end{assumption}

\corr{Note that in the rest of the section for sake of simplicity we
  shall refer to $R(x,{\bf y})$ or $R(x,{\bf z})$ as $R({\bf y})$ or
  $R({\bf z})$ and similarly for $F(x,\by)$ unless emphasis is
  needed.} We shall now prove several lemmas that will be useful to
prove the main results (Theorem
\ref{analyticity:theorem1}). 

\corg{Note that we apply many singular value inequalities and
  properties to prove Lemmas \ref{analyticity:lemma2},
  \ref{analyticity:lemma3} and parts of section \ref{erroranalysis},
  that can be found in \cite{Jabbari}. We recall some of the most
  commonly used. If $A,B \in \C^{n,n}$ then} \corr{
\begin{enumerate}
\item $\sigma_{max}(A + B) \leq \sigma_{max}(A) + \sigma_{max}(B)$.
\item $\sigma_{min}(A + B) \geq \sigma_{min}(A) - \sigma_{max}(B)$.
\item $\sigma_{max}(AB) \leq \sigma_{max}(A)\sigma_{max}(B)$.
\item $\sigma_{min}(AB) \geq \sigma_{min}(A)\sigma_{min}(B)$.
\end{enumerate}
In addition, if $A,B$ are Hermitian then
\begin{enumerate}
\item $\lambda_{max}(A + B) \leq \lambda_{max}(A) + \lambda_{max}(B)$
  (Lidskii inequality).
\item $\lambda_{min}(A + B) \geq \lambda_{min}(A) + \lambda_{min}(B)$
  (Dual Lidskii inequality).
\end{enumerate}
}

\begin{lem}
Under Assumption \ref{analyticity:assumption1} we have that $\forall
\by \in [-1,1]^{N}$ and $x \in U$
\begin{enumerate}[i)]
\item $\sigma_{max}(\partial F(\by)) \leq 2 - \tilde{\delta}$,
\item $\sigma_{min}(\partial F(\by)) \geq \tilde{\delta}$,
\item $(2 - \tilde{\delta})^{d} \geq det(\partial F(\by)) \geq \tilde{\delta}^{d}$.
\end{enumerate}
\label{analyticity:lemma0}
\end{lem}
\begin{pf}
\begin{enumerate}[i)]
\item $\| \partial F(\by) \|_2 \leq 1 + \sup_{x \in U} 
\sum_{l=1}^{N} \| B_l(x) \|_{2} \sqrt{\mu_{l}}
 \leq 2 - \tilde{\delta}$.
\item $\sigma_{max}(\partial F(\by) - I) = \| \sum_{l=1}^{N} B_l(x) 
\sqrt{\mu_{l}}y_l \|_2
\leq 1 - \tilde{\delta} \Rightarrow
\sigma_{min}(\partial F(\by)) 
= \sigma_{min}(I + \partial F(\by) - I) \geq \sigma_{min}(I)
-\sigma_{max}(\partial F(\by) - I) \geq 1 - (1-\tilde{\delta}) 
= \corm{\tilde{\delta}}.$ 
\item The result follows from the following fact: If $A \in \C^{d
    \times d}$ we have that $\sigma_{min}(A) \leq | \lambda_{l}(A)|
  \leq \sigma_{max}(A)$ for all $l = 1, \dots, d$ (see \cite{Jabbari}).
\end{enumerate}
\end{pf}

\begin{lem} Let $0 < \beta < \tilde{\delta} \frac{\log{2}}{d + \log{2}}$ 
and $\alpha = 2 - exp(\frac{d \beta }{\tilde{\delta} - \beta}) > 0$
then $\forall {\bf z} \in \Theta_{\beta}$ and $\forall x \in U$ we
have that $det(\partial F({\bf z}))$ is analytic and
\begin{enumerate}[i)]
\item $|det(\partial F({\bf z})) 
| \geq \tilde{\delta}^{d} \alpha$,
\item $|det(\partial F({\bf z})) | \leq (2 - \tilde{\delta})^{d} 
(2 - \alpha)
$,
\item $\Real det(\partial F({\bf z})) \geq \tilde{\delta}^{d} \alpha$,
  $|\Imag det(\partial F({\bf z}))| \leq (2-\tilde{\delta})^{d} (1 -
  \alpha)$.
\end{enumerate}
\label{analyticity:lemma1}
\end{lem}
\begin{pf}
For all ${\bf z} \in \Theta_{\beta}$ we have that
\[
\partial F(x,{\bf z}) = I + 
\sum_{l=1}^{N} B_{l}(x) \sqrt{\mu_{l}} y_{l} +
\sum_{l=1}^{N} B_{l}(x) \sqrt{\mu_{l}} w_{l}
= 
I + R({\bf y}) + R({\bf w})
\]
and let $Q({\bf y},{\bf w}) = I + \partial F({\bf y}) ^{-1} R({\bf
  w})$ so that $\partial F({\bf z})= \partial F({\bf y})Q({\bf y},{\bf
  w})$.

We now study $det(Q({\bf y},{\bf w}))$ for all ${\bf z} \in
\Theta_{\beta}$ by using the following identity \cite{Ipsen2011}: If
$A \in \C^{d \times d}$ and $\sigma_{max}(A) < 1$ then
\[
det(I + A) = 
1 + \corm{\sum_{k=1}^{\infty} \frac{1}{k!}  \left( -\sum_{j=1}^{\infty}
  \frac{(-1)^{j}}{j} tr(A^{j}) \right)^{k}.}
\]
It follows that
\[
\begin{split}
|det(Q({\bf y},{\bf w})) - 1|
&\leq \sum_{k=1}^{\infty} \frac{1}{k!}  \left( \sum_{j=1}^{\infty}
  \frac{1}{j} |tr( ( \partial F(\by)^{-1} R(\bw)
)^{j})| \right)^{k}\\
&\leq \sum_{k=1}^{\infty} \frac{1}{k!}  \left( \sum_{j=1}^{\infty}
  \frac{1}{j} \sum_{l = 1}^{d}| \lambda_{l}( \partial F(\by)^{-1} R(\bw)
)|^{j}
 \right)^{k}\\
&\leq \sum_{k=1}^{\infty} \frac{1}{k!}  \left( \sum_{j=1}^{\infty}
  \frac{1}{j} d \sigma^{-j}_{min}(\partial F(\by)) \sigma^{j}_{max}(R(\bw))
\right)^{k}\\
&\leq \sum_{k=1}^{\infty} \frac{1}{k!}  \left( \sum_{j=1}^{\infty}
  d (\beta / \tilde{\delta} )^{j}
\right)^{k} \\
&(\mbox{From Lemma \ref{analyticity:lemma0}}) \\
&= \sum_{k=1}^{\infty} \frac{1}{k!}  \left( 
  \frac{d (\beta / \tilde{\delta} )}{1 - \beta / \tilde{\delta} }
\right)^{k}= \exp \left( \frac{d \beta }{\tilde{\delta} - \beta} \right) - 1,
\end{split}
\]
and $|det(Q({\bf y},{\bf w}))| \leq \exp \left( \frac{d
  \beta}{\tilde{\delta} - \beta} \right)$. 

Now, it follows that if $\beta < \frac{\tilde{\delta} \log{2}}{d +
  \log{2}}$ then
\begin{enumerate}[i)]
\item $|det(Q({\bf y},{\bf w}))| \geq 1 - |det(Q({\bf y},{\bf w})) - 1| \geq 2 - 
\exp \left( \frac{d \beta}{\tilde{\delta} -
    \beta} \right)  = \alpha > 0$,
\item $\Real det(Q({\bf y},{\bf w})) \geq 1 - |det(Q({\bf y},{\bf
    w})) - 1| \geq \alpha > 0$,
\item $|\Imag det(Q({\bf y},{\bf w}))| \leq |det(Q({\bf y},{\bf w})) -
  1| \leq 1 - \alpha$.
\end{enumerate}
Finally we have that $det(\partial F(\bz)) = det(\partial F(\by))
det(Q(\by,\bw)) $. \corr{It is easy now to see that $det(\partial F(\bz))$
is analytic $\forall \bz \in \Theta_{\beta}$ since $det(\partial
F(\bz))$ is a finite polynomial of $\bz$.} The rest of the result
follows by applying Lemma \ref{analyticity:lemma0}.
\end{pf}



\begin{lem} 
Let $G({\bf z}) := (a \circ F) det( \partial F({\bf z})) \partial
F^{-1}({\bf z}) \partial F^{-T}({\bf z})$ and suppose
\begin{equation}
0 < \beta < \min \corm{\Bigg\{ \frac{ \tilde{\delta} \log{(2 -
      \gamma)}}{d + \log{ (2 - \gamma) }}, \sqrt{1 +
    \tilde{\delta}^2/2} - 1 \Bigg\},}
\label{analyticity:eqn9}
\end{equation}
where \corr{$\gamma := \frac{2\tilde{\delta}^2 +
    (2-\tilde{\delta})^{d}}{\tilde{\delta}^{d} +
    (2-\tilde{\delta})^{d}}$} then \corred{$\forall x \in U$ we have
  that} $\Real G({\bf z})$ is positive definite $\forall \bz \in
\Theta_{\beta}$ and

\begin{enumerate}[(a)]
\item $\lambda_{min}(\Real G(\bz)^{-1}) \geq
  B(\tilde{\delta},\beta,d,a_{max})>0$\corm{, where}
\[
\corr{
B(\tilde{\delta},\beta,d,a_{max}) 
:=
\frac{\tilde{\delta}^{d+1}\alpha 
( \tilde \delta - 2 \beta
) - 2\beta(2 + (\beta - \tilde{\delta}))
(1 - \alpha)(2 - \tilde{\delta})^{d}
}
{a_{max}(2 - \tilde{\delta})^{2d} (2 - \alpha)^{2}}. 
}
\]
\item $\lambda_{max}(\Real G(\bz)^{-1}) \leq D(\tilde{\delta},\beta,d,a_{min})
  < \infty$\corm{, where}
\[
\begin{split}
D(\tilde{\delta},\beta,d,a_{min}) 
&:=\frac{1}{a_{min}
\tilde{\delta}^{2d}\alpha^{2}}
\left[(2 - \tilde{\delta})^{d} (2 - \alpha)
\corr{(2 - \tilde{\delta} + \beta)^{2}}
\right. \\
&\left.+ 2(1 - \tilde{\delta})^{d} 
(2 - \alpha)\beta(2 + (\beta - \tilde{\delta}))\right].
\end{split}
\]
\item $\sigma_{max}(\Imag G(\bz)^{-1}) \leq
  C(\tilde{\delta},\beta,d,a_{min}) < \infty$\corm{, where}
\[
\begin{split}
C(\tilde{\delta},\beta,d,a_{min}) 
&:=\frac{1}{a_{min}
\tilde{\delta}^{2d}\alpha^{2}}
\left[(2 - \tilde{\delta})^{d} (2 - \alpha)
2\beta(2 + (\beta - \tilde{\delta}))
\right. \\
& \left. + (2 - \tilde{\delta})^{d} (1 - \alpha)
(((2-\tilde{\delta}) + \beta)^{2} + \beta^{2}) \right].
\end{split}
\]
\end{enumerate}
\label{analyticity:lemma2}
\end{lem}
\begin{pf}
{\bf (a)} To simplify the proof we use the property that if $\Real
G^{-1}({\bf z})$ is positive definite then $\Real G({\bf z})$ is
positive definite (From (b) in \cite{london1981}) , but first we
derive bounds for $\Real \partial F({\bf z})^{T} \partial F({\bf z}) $
and $\Imag \partial F({\bf z})^{T} \partial F({\bf z})$. Now,
\corred{$\forall x\in U$ and $\forall \bz \in \Theta_{\beta}$}
we have that
\[
\begin{split}
  \Real \partial F({\bf z})^{T} \partial F({\bf z}) & = 
\corr{\Real [(I +  R({\bf y}) + R({\bf w})
)^{T}(I + R({\bf y}) + R({\bf w})
  )] }
\\
  & = 
(I + R({\bf y}) + R_{r}({\bf w})
)^{T}(I + R({\bf y}) + R_{r}({\bf w})
) -
  \corm{R_{i}({\bf w})^{T}R_{i}({\bf w}),}
\end{split}
\]
where $R({\bf w})= R_{r}({\bf w}) + iR_{i}({\bf w})$. By applying the
dual Lidskii inequality (if $A,B \in \mathbb{C}^{d \times d}$ are
Hermitian then $\lambda_{min}(A+B) \geq \lambda_{min}(A) +
\lambda_{min}(B)$) we obtain
\begin{equation}
\begin{split}
  \lambda_{min}(\Real \partial F({\bf z})^{T} \partial F({\bf z})) 
&
  \geq \lambda_{min}(
(I + R({\bf y}) + R_{r}({\bf w})
)^{T}(I + R({\bf y}) + R_{r}({\bf w})
)) \\
&- \lambda_{max}(
  R_{i}({\bf w})^{T}R_{i}({\bf w})) \\
&=
\sigma^{2}_{min}(
I + R({\bf y}) + R_{r}({\bf w})
)
- \sigma^{2}_{max}(
  R_{i}({\bf w})) \\
&\geq
(\sigma_{min}(
I + R({\bf y})) - 
\sigma_{max}( 
R_{r}({\bf w})
))^{2}
- \sigma^{2}_{max}(
  R_{i}({\bf w})) 
\\
&\geq
(\tilde{\delta} - \beta)^{2} - \beta^{2}.
\end{split}
\label{analyticity:eqn5}
\end{equation}
It follows that if $\beta < \tilde{\delta}/2$ then \corred{$\forall
  x\in U$ and $\forall \bz \in \Theta_{\beta}$}
\[
\begin{split}
  \lambda_{min}(\Real \partial F({\bf z})^{T} \partial F({\bf z})) 
&\geq \tilde{\delta}(\tilde{\delta} - 2 \beta) > 0.
\end{split}
\]
and is positive definite. We see that \corred{$\forall x\in U$ and
  $\forall \bz \in \Theta_{\beta}$},
\begin{equation}
\begin{split}
\corr{\max_{i = 1,\dots,d}|\lambda_{i}(\Imag \partial F({\bf
  z})^{T} \partial F({\bf z}))|}
& \leq
\corr{\sigma_{max}(
R_{i}({\bf w})^{T}
(I + R({\bf y}) + R_{r}({\bf w})) } \\
&+
(I + R({\bf y}) + R_{r}({\bf w}))^{T}
R_{i}({\bf w})) \\
& \leq
2 \sigma_{max} (R_{i}({\bf w}))
\sigma_{max} (I + R({\bf y}) + R_{r}({\bf w})) \\
& \leq  2 \beta(2 + (\beta - \tilde{\delta})).
\end{split}
\label{analyticity:eqn6}
\end{equation}
We now have that 
\corred{$\forall x\in U$ and $\forall \bz \in \Theta_{\beta}$}
\begin{equation}
\begin{split}
\corr{\lambda_{min}( (a \circ F) ^{-1}} \Real (\xi^{-1}({\bf z}) 
\partial
F({\bf z}) ^{T} \partial F({\bf z}) ) ) 
&\geq \frac{1}{a_{max}|\xi({\bf z})|^{2}}  \lambda_{min}(
\xi_{R}(\bz)
 \Real F({\bf z}) ^{T} \partial F({\bf z}) \\
&+
\xi_{I}(\bz)
 \Imag F({\bf z}) ^{T} \partial F({\bf z}) )
\\
&\geq 
\frac{1}{a_{max}|\xi({\bf z})|^{2}}(  
\xi_{R}(\bz) 
\lambda_{min}( \Real \partial F({\bf z})^{T} \partial F({\bf z}) ) \\
&- |\xi_{I}(\bz)|
\corr{|\max_{i = 1,\dots,d} \lambda_{i}( \Imag \partial F({\bf z})^{T} \partial F({\bf z}) )| )},
\end{split}
\label{analyticity:eqn10}
\end{equation}
\noindent where $\xi({\bf z}) := \xi_{R}({\bf z}) + i\xi_{I}({\bf z})
= det(I + R({\bf z}) ) $. From Lemma \ref{analyticity:lemma2}
\corred{$\forall x\in U$ and $\forall \bz \in \Theta_{\beta}$} we have
that $|\xi({\bf z}) |^{-1} \geq (2 - \tilde{\delta})^{-d} (2 -
\alpha)^{-1} > 0$ whenever ${\bf z} \in \Theta_{\beta}$ \corred{and
  thus}
\begin{equation}
\begin{split}
\corr{\max_{i = 1,\dots,d}|\lambda_{i}(\Imag \partial F({\bf
  z})^{T} \partial F({\bf z}))|}
& \leq
\corr{\max_{i = 1,\dots,d}|\lambda_{i}}
(R_{i}({\bf w})^{T}
(I + R({\bf y}) + R_{r}({\bf w})) \\
& +
(I + R({\bf y}) + R_{r}({\bf w}))^{T}
R_{i}({\bf w}))|
\\
& \leq  2 \beta(2 + (\beta - \tilde{\delta})).
\end{split}
\label{analyticity:eqn13}
\end{equation}

From Lemma \ref{analyticity:lemma0} (iii) if \corr{$\beta <
  \frac{\tilde{\delta} \log{ \gamma }}{d + \log{ \gamma }}$} , $\gamma
:= \frac{2 \tilde{\delta}^d +
  (2-\tilde{\delta})^{d}}{\tilde{\delta}^{d} +
  (2-\tilde{\delta})^{d}}$, then $\xi_R(\bz) > |\xi_I(\bz)|$, $\forall
\bz \in \Theta_{\beta}$ and \corred{$\forall \bz \in \Theta_{\beta}$}.  From
inequalities \eqref{analyticity:eqn5} and \eqref{analyticity:eqn6} we
have that if $\beta < \sqrt{1 + \tilde{\delta}^2/2} - 1$ then
\corred{$\forall x\in U$ and $\forall \bz \in \Theta_{\beta}$}
\[
\corr{ \lambda_{min}(\Real \partial F({\bf z})^{T} \partial F({\bf z}))
  > \max_{i = 1,\dots,d} |\lambda_{i}(\Imag \partial F({\bf z})^{T}
  \partial F({\bf z}))|}
\]
and $\lambda_{min}(\Real G(\bz)^{-1}) \geq B$
$(\tilde{\delta},\beta,d,a_{max})>0$\corm{, where}
\[
\corr{ 
B(\tilde{\delta},\beta,d,a_{max}) 
:=
\frac{\tilde{\delta}^{d+1}\alpha 
( \tilde \delta - 2 \beta
) - 2\beta(2 + (\beta - \tilde{\delta}))
(1 - \alpha)(2 - \tilde{\delta})^{d}
}
{a_{max}(2 - \tilde{\delta})^{2d} (2 - \alpha)^{2}}}.
\]
From London's Lemma \cite{london1981} it follows that $\Real G({\bf
  z})$ is positive definite $\forall \bz \in \Theta_{\beta}$.  

{\bf (b)} By applying the Lidskii inequality (If $A,B \in
\mathbb{C}^{d \times d}$ are Hermitian then $\lambda_{max}(A+B) \leq
\lambda_{max}(A) + \lambda_{max}(B)$) we have that \corred{$\forall
  x\in U$ and $\forall \bz \in \Theta_{\beta}$}
\begin{equation}
\begin{split}
  \lambda_{max}(\Real \partial F({\bf z})^{T} \partial F({\bf z})) 
&
  \leq \lambda_{max}(
(I + R({\bf y}) + R_{r}({\bf w})
)^{T}(I + R({\bf y}) + R_{r}({\bf w})
)) \\
&- \lambda_{min}(
  R_{i}({\bf w})^{T}R_{i}({\bf w})) \\
&=
\sigma^{2}_{max}(
I + R({\bf y}) + R_{r}({\bf w})
)
- \sigma^{2}_{min}(
  R_{i}({\bf w})) \\
&\leq
(\sigma_{max}(
I + R({\bf y})) +
\sigma_{max}( 
R_{r}({\bf w})
))^{2}
\\
&\leq
\corr{(2 - \tilde{\delta} + \beta)^{2}}.
\end{split}
\label{analyticity:eqn12}
\end{equation}
From inequalities \eqref{analyticity:eqn12} and \eqref{analyticity:eqn13},
and Lemmas \ref{analyticity:lemma0} and \ref{analyticity:lemma1},
we obtain \corred{$\forall x\in U$ and $\forall \bz \in \Theta_{\beta}$}
\[
\begin{split}
\lambda_{max}(\Real G(\bz)^{-1}) 
&\leq
\frac{
|\xi_{R}(\bz)| 
\lambda_{max}(
\Real \partial F({\bf z})^{T} \partial F({\bf z}) ) 
+ |\xi_{I}(\bz)|
\corr{|\max_{i}|\lambda_{i}}( 
\Imag \partial F({\bf z})^{T} \partial F({\bf z}) )|
}{a_{min}|\xi({\bf z}) |^{2}} \\
&\leq D(\tilde{\delta},\beta,d,a_{min}) < \corm{\infty,}
\end{split}
\] 
where
\[
\begin{split}
D(\tilde{\delta},\beta,d,a_{min}) 
&:=\frac{1}{a_{min}
\tilde{\delta}^{2d}\alpha^{2}}
\left[(2 - \tilde{\delta})^{d} (2 - \alpha)\corr{(2 
- \tilde{\delta} + \beta)^{2}}
\right. \\
&\left.+ 2(2 - \tilde{\delta})^{d} 
(1 - \alpha)\beta(2 + (\beta - \tilde{\delta}))\right].
\end{split}
\]

{\bf (c)} Similarly, \corred{$\forall x\in U$ and $\forall \bz \in
  \Theta_{\beta}$}, we can bound
\begin{equation}
\begin{split}
\sigma_{max}(\Imag \partial F({\bf
  z})^{T} \partial F({\bf z}))
& \leq
\sigma_{max}(
R_{i}({\bf w})^{T}
(I + R({\bf y}) + R_{r}({\bf w})) \\
& +
(I + R({\bf y}) + R_{r}({\bf w}))^{T}
R_{i}({\bf w}))
\\
& \leq
2\sigma_{max}(
R_{i}({\bf w}))
\sigma_{max}(I + R({\bf y}) + R_{r}({\bf w}))
\\
& \leq  2 \beta(2 + (\beta - \tilde{\delta})).
\end{split}
\label{analyticity:eqn14}
\end{equation}
and
\begin{equation}
\begin{split}
  \sigma_{max}(\Real \partial F({\bf z})^{T} \partial F({\bf z})) 
&
  \leq \sigma_{max}(
(I + R({\bf y}) + R_{r}({\bf w})
)^{T}(I + R({\bf y}) + R_{r}({\bf w})
)) \\
&+ \sigma_{max}(
  R_{i}({\bf w})^{T}R_{i}({\bf w})) \\
&=
\sigma^{2}_{max}(
I + R({\bf y}) + R_{r}({\bf w})
)
+ \sigma^{2}_{max}(
  R_{i}({\bf w})) \\
&\leq
(\sigma_{max}(
I + R({\bf y})) +
\sigma_{max}( 
R_{r}({\bf w})
))^{2}
+ \sigma^{2}_{max}(
  R_{i}({\bf w})) 
\\
&\leq
((2-\tilde{\delta}) + \beta)^{2} + \beta^{2}.
\end{split}
\label{analyticity:eqn15}
\end{equation}
From inequalities \eqref{analyticity:eqn14} and \eqref{analyticity:eqn15},
and Lemmas \ref{analyticity:lemma0} and \ref{analyticity:lemma1} we obtain
\corred{$\forall x\in U$ and $\forall \bz \in  \Theta_{\beta}$}
\[
\begin{split}
\sigma_{max}(\Imag G(\bz)^{-1}) 
&\leq
\frac{
\sigma_{max}(
\xi_{R}(\bz) 
\Imag \partial F({\bf z})^{T} \partial F({\bf z}) 
- \xi_{I}(\bz) 
\Real \partial F({\bf z})^{T} \partial F({\bf z}))
}{a_{min}|\xi({\bf z}) |^{2}} \\
&\leq 
\frac{
|\xi_{R}(\bz)|
\sigma_{max}(
\Imag \partial F({\bf z})^{T} \partial F({\bf z}) ) 
+ |\xi_{I}(\bz)| 
\sigma_{max}( \Real \partial F({\bf z})^{T} \partial F({\bf z}) )
}{a_{min}|\xi({\bf z}) |^{2}}
\\
& \leq  C(\tilde{\delta},\beta,d,a_{min}) < \infty,
\end{split}
\]
where
\[
\begin{split}
C(\tilde{\delta},\beta,d,a_{min}) 
&:=\frac{1}{a_{min}
\tilde{\delta}^{2d}\alpha^{2}}
\left[(2 - \tilde{\delta})^{d} (2 - \alpha)
2\beta(2 + (\beta - \tilde{\delta}))
\right. \\
& \left. + (2 - \tilde{\delta})^{d} (1 - \alpha)
(((2-\tilde{\delta}) + \beta)^{2} + \beta^{2}) \right].
\end{split}
\]
\end{pf}
\begin{lem} 
\corr{For all $\bz \in \Theta_{\beta}$ \corred{and $x \in U$}}
\[
\lambda_{min}(\Real G(\bz)) \geq
\varepsilon(\tilde{\delta},\beta,d,a_{max},a_{min}) \corm{> 0,}
\]
where
\begin{equation}
\varepsilon(\tilde{\delta},\beta,d,a_{max},a_{min}) := \frac{1}
           {\left(1 + \left(
             \frac{C(\tilde{\delta},\beta,d,a_{min})}{B(\tilde{\delta},\beta,d,a_{max})}
             \right) ^{2} \right)D(\tilde{\delta},\beta,d,a_{min}) }.
\label{analyticity:eqn16}
\end{equation}
\label{analyticity:lemma3}
\end{lem}
\begin{pf} From Lemma \ref{analyticity:lemma2} $\Real G(\bz)$ is 
positive definite \corred{$\forall x\in U$ and $\forall \bz \in
  \Theta_{\beta}$}, where $\beta$ satisfies
\eqref{analyticity:eqn9}. It follows from the Lemma in
\cite{london1981} that $G(\bz) = Q(I +i \Lambda)Q^{*}$, where $Q$ is a
non-singular matrix, $\Lambda:=diag(\alpha_1,\dots,\alpha_d)$ and
$\alpha_{1},\dots,\alpha_{d}$ are real. Since $G(\bz)$ is symmetric
then $\Real G(\bz) = (1/2)(G(\bz) + G(\bz)^{*})$ and it is simple to
see that $\Real G(\bz) = QQ^{*}$. \corr{Note that in
  \cite{london1981}, the real part of a generic complex matrix $A$
  (i.e. $(1/2)(A + A^{*})$) is not in general $\Real A$.}

We need now to show that \corred{$\forall x\in U$ and $\forall \bz \in
  \Theta_{\beta}$} $\lambda_{min}(\Real G(\bz)) = \sigma^{2}_{min}(Q)
\geq \varepsilon > 0$, \corr{with $\varepsilon$ defined in equation
  \eqref{analyticity:eqn16}}.  Applying (b) in \cite{london1981} we
have that
\[
G(\bz)^{-1} = (DQ^{-1})^{*}
(I - i \Lambda) 
D\corm{Q^{-1},}
\]
where $D := diag((1 + \alpha_1^2)^{-1/2}, \dots, (1 +
\alpha_d^2)^{-1/2}) = (I + \Lambda^{2})^{-1/2}$.  It follows that
$\Real G(\bz)^{-1} = (DQ^{-1})^{*}DQ^{-1}$ \corred{$\forall x\in U$
  and $\forall \bz \in \Theta_{\beta}$},
\[
\lambda_{max}(\Real G(\bz)^{-1})
= 
\sigma^{2}_{max}(DQ^{-1}) \geq
 \sigma^{2}_{min}(D) 
\sigma^{2}_{max}(Q^{-1}) 
= \sigma^{2}_{min}(D) 
\sigma^{-2}
_{min}(Q),
\]
and therefore 
\[
\sigma^{2}_{min}(Q) 
\geq
\corr{
\frac{\sigma^{2}_{min}(D)}{\lambda_{max}(\Real G(\bz)^{-1})} 
= 
\frac{\sigma_{min}((I + \Lambda^{2})^{-1})}
{\lambda_{max}(\Real G(\bz)^{-1}) } 
\geq
\frac{(1 + \sigma^{2}_{max}(\Lambda))^{-1}}
{\lambda_{max}(\Real G(\bz)^{-1})}}.
\]
Now, $\Imag G(\bz)^{-1} = (DQ^{-1})^{*}(-\Lambda)DQ^{-1}$ and
\corred{$\forall x\in U$ and $\forall \bz \in \Theta_{\beta}$}
\[
\sigma_{max}(\Imag G(\bz)^{-1}) 
\geq \sigma^{2}_{min}(DQ^{-1}) \sigma_{max}
(\Lambda).
\]
Since $\Real G(\bz)^{-1} = (DQ^{-1})^{*}DQ^{-1}$ then
\corred{$\forall x\in U$ and $\forall \bz \in  \Theta_{\beta}$}
$\lambda_{min}(\Real G(\bz)^{-1}) = \sigma^{2}_{min}(DQ^{-1})$ and
\[
A(\tilde{\delta},\beta,d,a_{max},a_{min}) : = \frac{\sigma_{max}(\Imag G(\bz)^{-1})
}{\lambda_{min}(\Real G(\bz)^{-1})
} \geq
\sigma_{max}(\Lambda).
\]
It follows that \corred{$\forall x\in U$ and $\forall \bz \in
  \Theta_{\beta}$}
\begin{equation}
\lambda_{min}(\Real G(\bz)) 
\geq
\frac{1}
{(1 + A^{2})|\lambda_{max}(\Real G(\bz)^{-1})| }.
\label{analyticity:eqn11}
\end{equation}
From Lemma \ref{analyticity:lemma2} (a) we have that
$\lambda_{min}(\Real G(\bz)^{-1}) \geq
\corr{B(\tilde{\delta},\beta,d,a_{max})>0}$. From Lemma
\ref{analyticity:lemma2} (c) we have that $\sigma_{max}(\Imag
G(\bz)^{-1}) \leq C(\tilde{\delta},\beta,d,a_{min}) < \infty$. This
implies $\sigma_{max}(\Lambda) \leq
A(\tilde{\delta},\beta,d,a_{max},a_{min}) < \infty$.  Finally from
Lemma \ref{analyticity:lemma2} (b) $\lambda_{max}(\Real G(\bz)^{-1})
\leq D(\tilde{\delta},\beta,d,a_{min})$ $< \infty.$ We conclude that
\corred{$\forall x\in U$ and $\forall \bz \in  \Theta_{\beta}$}
\[
\lambda_{min}(\Real G(\bz)) 
\geq
\varepsilon(\tilde{\delta},\beta,d,a_{max},a_{min}) > 0 .
\]
\end{pf}


We are now ready to prove the main result of this section. 



\begin{thm} Let $0 < \tilde{\delta} < 1$ \corg{
then the solution $\tilde{u} \circ F:\Gamma \rightarrow H^{1}_{0}(U)$
of Problem \ref{setup:Prob3}} can be extended holomorphically in
  $\Theta_{\beta} \subset \C^{N}$ (see equation
  \eqref{analyticity:region}) if
\[
\beta < min \left\{ \tilde{\delta} \frac{\log{(2 - \gamma)}}{d + \log{(2 -
    \gamma)}}, \sqrt{1 + \tilde{\delta}^2/2} - 1 \right\},
\]
where \corr{$\gamma := \frac{2\tilde{\delta}^2 +
    (2-\tilde{\delta})^{d}}{\tilde{\delta}^{d} +
    (2-\tilde{\delta})^{d}}$}. 
\label{analyticity:theorem1}
\end{thm}

\begin{pf}
The strategy for this proof \corm{is to show} that the function \corg{$(\tilde u
  \circ F)(\cdot,\by)$ can be extended on $\Theta_{\beta}$ and is
  analytic in each variable separately. Then apply Hartog's Theorem}
(Chap1, p32, \cite{Krantz1992}) and Osgood's Lemma (Chap 1, p 2,
\cite{Gunning1965}) to show that such an extension is analytic in
$\Theta_{\beta}$.

\corg{For $n = 1, \dots, N$ consider the map \cor{$\Psi(\by): \Gamma
    \rightarrow H^{1}_{0}(U)$} where
\[
\Psi(\by) := (\tilde u \circ F)(\by),
\]
\noindent for any arbitrary points $\by \in \Gamma$. Furthermore,
consider the extension of $\by \rightarrow \bz$, where $\bz \in
\C^{N}$.}

Since $\beta < \tilde{\delta}$ the series
\[
\partial F^{-1}({\bf z}) = (I + R({\bf z}) )^{-1} = I +
\sum_{k=1}^{\infty} R({\bf z})^{k}
\]
is convergent $\forall {\bf z} \in \Theta_{\beta}$. It follows that
each entry of $\partial F({\bf z})^{-1}$ is analytic \cor{in $
  \Theta_{\beta}$}. From Lemma \ref{analyticity:lemma1} it follows
that the entries of $G({\bf z})$ are analytic \corg{on
  $\Theta_{\beta}$}.

\corg{Let $\bar{\Psi}(\bz) := [\Psi_{R}(\bz) , \,\,\,
    \Psi_{I}(\bz)]^{T}$, with $\Psi_{R}(\bz) := Re\,\, \Psi(\bz)$ and
  $\Psi_{I}(\bz) := Im\,\, \Psi(\bz) $, be the solution (in the weak
  sense) of the problem
\begin{equation}
-\nabla \cdot \hat{G}(\bz) \nabla \bar \Psi(\bz) = \hat{f}(\bz),
\label{analyticity:eqn2}
\end{equation}
\noindent where 
\[
\hat G (\bz) := \left(
\begin{array}{cc}
G_{R}(\bz) &  -G_{I}(\bz)\\
G_{I}(\bz) &   G_{R}(\bz)\\
\end{array}
\right),\,\,\,
\hat{f} := \left(
\begin{array}{c}
\tilde f_{R}(\bz) \\
\tilde f_{I}(\bz) \\
\end{array}
\right),
\]
$G_{R}(\bz):= \Real(G(\bz))$, $G_{I}(\bz):=\Imag(G(\bz))$, $\tilde
f_R(\bz) := \Real \tilde{f}(\bz)$ and $\tilde f_I = \Imag
\tilde{f}(\bz)$. Note that $\tilde{f}$(\bz) refers to extension of the
right hand side of the weak formulation i.e.  $\tilde{l}(\bz;v)$ for
all $v \in H^1_{0}(U)$.  Thus $\tilde f_R(\bz):= \Real \{ (f \circ
F)(\cdot,\bz) |\partial F(\bz)| + \nabla \cdot (a \circ F) C^{-1}(\bz)
|\partial F(\bz)| \corred{\nabla \hat \bw} \}$ and
similarly for $\tilde f_I(\bz)$.}

The system of equations (\ref{analyticity:eqn2}) has a unique
  solution. Indeed, from Lemma \ref{analyticity:lemma3} we have that
  $\lambda_{min}(G_{R}({\bf z})) > 0$ $\forall \bz \in \Theta_{\beta}$
  \corred{$\forall x \in U$}.  Since $G_{I}({\bf z})$ is symmetric, it
  follows that $G({\bf z})$ is positive definite \corred{$ \forall x
    \in U$}, hence the well posedness of problem
  \eqref{analyticity:eqn2} by Lax-Milgram.  Moreover, the solution
  $\Psi(\bz)$ coincides with $(\tilde u \circ F)(\by)$, $\by \in
  \Gamma$, when $\bz \in \Gamma$, therefore it is a complex
  continuation of $\tilde u \circ F$ on $\Theta_{\beta}$.

\corg{To show that $\Psi(\bz): \Theta_{\beta} \rightarrow
  H^{1}_{0}(U)$ is holomorphic on $\Theta_{\beta}$ we focus on the
  $n^{th}$ variable $z_n$, $n = 1,\dots,N$, which we write as $z_n = s
  + iw$, $s,w \in \R$, and show that the Cauchy-Riemann conditions
are satisfied. But first we have to show that the derivatives
$\partial_{s} \Psi$ and $\partial_{w} \Psi$ exist.  Now,
differentiating (\ref{analyticity:eqn2}) with respect to $s= \Real
z_n$ and $w = \Imag z_n$ we obtain}
\begin{eqnarray}
- (\nabla \cdot G_{R} \nabla \partial_{s} \Psi_{R}\corg{(\bz)} 
- \nabla \cdot G_{I} \nabla \partial_{s} \Psi_{I}\corg{(\bz)})
&=&
\nabla \cdot \partial_{s} G_{R} \nabla \Psi_{R}\corg{(\bz)} 
- \nabla \cdot \partial_{s} G_{I} \nabla \Psi_{I}\corg{(\bz)} \nonumber \\ 
&+& \partial_{s} \tilde f_{R}\corg{(\bz)}, \nonumber \\
- (\nabla \cdot G_{I} \nabla \partial_{s} \Psi_{R}\corg{(\bz)} 
+ \nabla \cdot G_{R} \nabla \partial_{s} \Psi_{I}\corg{(\bz)})
&=&
\nabla \cdot \partial_{s} G_{I} \nabla \Psi_{R}\corg{(\bz)} 
+ \nabla \cdot \partial_{s} G_{R} \nabla \Psi_{I}\corg{(\bz)} \nonumber \\
&+& \partial_{s} \tilde f_{I}\corg{(\bz)}, \nonumber \\
- (\nabla \cdot G_{R} \nabla \partial_{w} \Psi_{R}\corg{(\bz)} 
- \nabla \cdot G_{I} \nabla \partial_{w} \Psi_{I}\corg{(\bz)})
&=&
\nabla \cdot \partial_{w} G_{R} \nabla \Psi_{R}\corg{(\bz)} 
- \nabla \cdot \partial_{w} G_{I} \nabla \Psi_{I}\corg{(\bz)} \nonumber \\
&+& \partial_{w} \tilde f_{R}\corg{(\bz)}, \nonumber \\
- (\nabla \cdot G_{I} \nabla \partial_{w} \Psi_{R}\corg{(\bz)} 
+ \nabla \cdot G_{R} \nabla \partial_{w} \Psi_{I}\corg{(\bz)})
&=&
\nabla \cdot \partial_{w} G_{I} \nabla \Psi_{R}\corg{(\bz)} 
+ \nabla \cdot \partial_{w} G_{R} \nabla \Psi_{I}\corg{(\bz)} \nonumber \\
&+& \partial_{w} \tilde f_{I}\corm{(\bz).} 
\label{analyticity:eqn7} 
\end{eqnarray}
\corg{Note that to avoid clutter in the equations we refer to
  $G_{R}(\bz)$ as $G_{R}$ and $G_{I}(\bz)$ as $G_{I}$.  By the
  Lax-Milgram theorem the derivatives $\partial_{s} \Psi\corg{(\bz)}$
  and $\partial_{w} \Psi\corg{(\bz)}$ exist and have a unique solution
  whenever $\bz \in \Theta_{\beta}$.}  The second step is now to show
that the Cauchy-Riemann conditions are satisfied.

\corg{Let $P(\bz) := \partial_{s} \Psi_{R}(\bz) - \partial_{w}
  \Psi_{I}(\bz)$ and $Q(\bz) := \partial_{w} \Psi_{R}(\bz) +
  \partial_{s} \Psi_{I}(\bz)$. To show analyticity we have to show
  that $P(\bz) = 0$ and $Q(\bz) = 0$ for all $\bz \in
  \Theta_{\beta}$.} By taking linear combinations of equations
\eqref{analyticity:eqn7} we obtain
\begin{eqnarray}
-\nabla \cdot (G_{R} \nabla P - G_{I} \nabla Q) &=& \nabla \cdot (
(\partial_{s} G_{R} - \partial_{w} G_{I}) \nabla \Psi_{R} -
(\partial_{w} G_{R} + \partial_{s} G_{I})\nabla \Psi_{I} )
\nonumber \\
&+& \partial_{s} \tilde f_{R} - \partial_{w} \tilde f_{I}, \nonumber \\
-\nabla \cdot (G_{I} \nabla P + G_{R} \nabla Q) &=& \nabla \cdot (
(\partial_{w} G_{R} + \partial_{s} G_{I}) \nabla \Psi_{R} -
(\partial_{s} G_{R} - \partial_{w} G_{I})\nabla \Psi_{I} ) \nonumber \\
&+& \partial_{s} \tilde f_{I} + \partial_{w} \corm{\tilde f_{R}.}
\label{analytic:reimann-cauchy}
\end{eqnarray}
\cor{We now need to show that $G\corg{(\bz)}$ and $\tilde
  f\corg{(\bz)}$ satisfy the Riemann-Cauchy conditions so that the
  right hand side becomes zero.}

From Assumption \ref{analyticity:assumption1} we have that \corred{$(f
  \circ F)(\bz)$ is analytic on $\Theta_{\beta}$} thus $\tilde
l(\bz;v)$ is holomorphic on $\Theta_{\beta}$.  Now, recall that
$G(\bz)$ is analytic if $\bz \in \Theta_{\beta}$.  Thus equations
\eqref{analytic:reimann-cauchy} have a unique solution \corg{$P(\bz) =
  Q(\bz) = 0$ for all $\bz \in \Theta_{\beta}$}


\corr{From Hartog's Theorem it follows that $\Psi({\bf z})$ is
  continuous on $\Theta_{\beta}$. From Osgood's Lemma it follows that
  $\Psi({\bf z})$ is holomorphic on $\Theta_{\beta}$.}
\end{pf}

\begin{corl}
\corg{The following estimate holds for all \corred{$\bz \in
    \Theta_{\beta}$:}
\begin{equation}
\| \nabla ((\tilde u \circ F)(\cdot,{\bf z})) \|_{L^{2}(U)} 
\leq 
\frac{E((f \circ F)(\bz), \corred{\hat \bw}, a_{max}, \tilde \delta,
\alpha, d, C_P(U))}
{\varepsilon(\tilde{\delta},\beta,d,a_{max},a_{min})},
\label{analyticity:eqn8}
\end{equation}
where
\[
\begin{split}
E((f \circ F)(\bz), \corred{\hat \bw}, a_{max}, \tilde \delta,
\alpha, d, C_P(U)) 
&:= 
(2 - \alpha) \Big( \frac{ 
\| \nabla  \corred{\hat \bw 
} \|_{L^2(U)}}{ a^{-1}_{max} (2 - \tilde \delta)^{-d} \tilde
  \delta^2}
\\
&+
C_{P}(U) \| (f \circ F)(\bz) \|_{L^2(U)} 
 \Big)
\end{split}
\]
and $\varepsilon(\tilde{\delta},\beta,d,a_{max},a_{min})$ is defined
in Lemmas \ref{analyticity:lemma2} and \ref{analyticity:lemma3}.}
\end{corl}
\begin{pf}
We formally \corred{multiply} \eqref{analyticity:eqn2} by $\bar
  \Psi({\bf z})^{T}$ and integrate over $U$ to obtain
\[
\int_{\corr{U}} (\nabla \Psi_{R}({\bf z}))^T G_{R}({\bf z}) \nabla
\Psi_{R}({\bf z}) + (\nabla \Psi_{I}({\bf z}))^T G_{R}({\bf z}) \nabla
\Psi_{I}({\bf z}) = 
\int_U \Psi_{R}({\bf z}) \tilde f_{R}(\bz) +
\Psi_{I}({\bf z}) \tilde f_{I}(\bz).
\]
From Lemma \ref{analyticity:lemma0}, Lemma \ref{analyticity:lemma1}
and the Poincar\'{e} inequality we have that
\[
\begin{split}
|\int_U \Psi_{R}(\bz) \tilde f_{R}(\bz) 
+
\Psi_{I}(\bz) \tilde f_{I}(\bz) | 
&\leq
|\int_U \Real \{ \Psi(\bz)^{*} \tilde f(\bz) \} |
\leq
\int_U | (\nabla \corred{\hat \bw})^T G(\bz)\nabla  \Psi(\bz)^{*} | \\
&+
\int_U | \Psi(\bz)^{*} (f \circ F)(\bz) | \partial F(\bz) || \\
&\leq \frac{ \| \nabla \corred{\hat \bw} \|_{L^2(U)} 
\| \nabla  \Psi(\bz) \|_{L^2(U)}}{ 
a^{-1}_{max} (2 - \tilde \delta)^{-d} (2 - \alpha)^{-1}  \tilde \delta^2} \\
&+
(2 - \tilde \delta)C_{P}(U)
\| \nabla  \Psi(\bz) \|_{L^2(U)}
\| (f \circ F)(\bz) \|_{L^2(U)}.
\end{split}
\]
It follows that
\[
\| \nabla \Psi({\bf z}) \|_{L^{2}(U)}
\min_{x \in U}
\corred{\{\lambda_{min}(G_{R}({\bf z}))\}}  
\leq E((f \circ F)(\bz), \corred{\hat \bw}, 
a_{max}, \tilde \delta, \alpha, d).
\]
From Lemma \ref{analyticity:lemma3} the result follows.
\end{pf}
\begin{remark} \corred{We can relax the restrictions  on
$f$ from Assumption \ref{analyticity:assumption2} for Theorem
    \ref{analyticity:theorem1}. For example, suppose that the
    $N_{\bq}$ valued random vector $\bq$ takes values on $
    \Gamma_{\bq}:= \tilde \Gamma_1 \times \dots \times \tilde
    \Gamma_{N_{q}}$ (with the probability density $\tilde
    \rho(\bq)$). Assume that the random vector $\bq$ is independent
    from $\by$ and write $f$ as}
\[
\corm{f(\cdot,\bq) = 
\sum_{j=1}^{N_f} \tilde{y}_{j}(\bq)
\tilde b_{j}\corred{(\cdot)},}
\]
where for $j = 1, \dots, N_f$, $\tilde{y}_{j} \in L^{\infty}_{\tilde
  \rho}(\Gamma_{\bq})$ and \corm{$\tilde b_j:\R^{d} \rightarrow
  \R$}. We now have stochastic contributions from the coefficients
$\tilde{y}_{j}(\bq)$. Since $\tilde b_j$ is defined on $\R^{d}$ we can
remap onto the reference domain and obtain
\[
\corg{(f \circ F)(\cdot,\by,\bq) = \sum_{j=1}^{N_f} \tilde{y}_{j}(\bq)
  (\tilde b_{j} \circ F)(\cdot,\by).}
\]
\corm{Now, assume that for $j = 1, \dots, N_f$, $\tilde b_{j}$ can be
  analytically extended in $\C^{d}$ . Furthermore, for $j = 1, \dots,
  N_f$ the extensions $\Real (\tilde b_j \circ F)(\cdot,\bz), \Imag
  (\tilde b_j \circ F)(\cdot,\bz) \in L^{2}(U)$ $\forall \bz \in
  \Theta_{\beta}$.}  This implies that we can holomorphically extend
$(\tilde b_{j} \circ F)(\by)$ into \corred{$\Theta_{\beta}$}. Now,
suppose that the coefficients $\tilde{y}_{j}(\bq)$ can be analytically
extended into $\C^{N_{\bq}}$.  It follows that $(f \circ
F)(\cdot,\by,\bq)$ can be analytically extended in
\corred{$\Theta_{\beta} \times \C^{N_{\bq}}$}.

\corred{Using a similar proof strategy (and linearity of the elliptic
operator) as in Theorem \ref{analyticity:theorem1}, we can show that
$(\tilde{u} \circ F)(\cdot,\by,\bq)$ can be analytically extended
along each separate dimension in $\Theta_{\beta} \times \C^{N_{\bq}}$.
By using Hartog's Theorem and Osgood's Lemma we can conclude that
$(\tilde{u} \circ F)(\cdot,\by,\bq)$ can be analytically extended in
$\Theta_{\beta} \times \C^{N_{\bq}}$}.
\label{analyticity:remark1}
\end{remark}

\section{Stochastic Collocation}
\label{stochasticcollocation}

We seek to efficiently approximate the mean and variance of the QoI of
the form \eqref{setup:qoi}. More specifically we seek a numerical
approximation to the exact moments of the QoI in a finite dimensional
subspace $ V_{{ p},h}$ based on a tensor product structure, where the
following hold:
\begin{itemize}
\item ${ H_h(U)}\subset H_0^1(U)$ is a standard finite element space of
  dimension $N_h$, which contains continuous piecewise polynomials
  defined on regular triangulations $\mathcal{T}_h$ that have a maximum mesh
  spacing parameter $h>0$.
\item ${ \mcP_{ p}(\Gamma)} \subset L^2_\rho(\Gamma)$ is the span of
  tensor product polynomials of degree at most $ {p =
    (p_1,\ldots,p_{N})}$; i.e., $ \mcP_{ p}(\Gamma) =
  \bigotimes_{n=1}^{N}\;\mcP_{ p_n}(\Gamma_{n})$ with
\begin{equation*}
\mcP_{ p_n}(\Gamma_{n})=\text{\rm span}(y_n^m,\,m=0,\dots,p_n),
\quad n=1,\dots,N.
\end{equation*}
Hence the dimension of $\mcP_p$ is $N_p=\prod_{n=1}^N (p_n+1)$.
\item 
$u_h : \Gamma \rightarrow H_h(U)$ is the semidiscrete approximation
  that is obtained by projecting the solution of \eqref{setup:Prob3}
  onto the subspace $H_h(U)$, for each $\by\in\Gamma$, i.e.,

\begin{equation}
\begin{split}
\int_U (a \circ F)(\cdot, \by)[\Grad u_h(\by)]^{T} C^{-1}(\by)
  \Grad v_h |\partial F|(\by)|
 \,dx 
&= \int_U (f \circ F)(\cdot,\by)
  v_h |\partial F|(\by)| \,dx \\
&- L(\hat \bw,v_h)
\end{split}
\label{collocation-perturbation:eq1}
\end{equation}

\corg{$\forall v_h\in H_h(U)$ and for a.e.  $\by\in\Gamma$.}  Denote
$\pi_{h}:H^{1}_{0}(U) \rightarrow H_{h}(U)$ as the finite element
operator s.t. if $u \in H^{1}_{0}(U)$ then $u_h := \pi_h u$ and
\begin{equation}
\| u - \pi_h u \|_{H^{1}_{0}(U)} \leq C_{\pi} \min_{v \in H_{h}(U)}
\| u - v \|_{H^{1}_{0}(U)} \leq h^{r}C(r,u).
\label{collocation:perturbation:finiteelement}
\end{equation}
\noindent The constant $r \in \mathbb{N}$ will depend on the
regularity of $u$ and the polynomial order of the finite element
space $H_{h}$.  Denote $C_{\Gamma}(r) : = \sup_{\by
  \in \Gamma} C(r,u(\by))$.

\item Similarly, $\varphi_{h} :=\pi_{h} \varphi$ is the semi-discrete
  approximation of the influence function. For each $\by\in\Gamma$,
  i.e.,
\begin{equation}
\corg{\int_U (a \circ F)(\cdot,\by)[\Grad v_h(\by)]^{T} C^{-1}(\by) \Grad
\varphi_h\,dx = Q(v_h) \quad \forall v_h\in H_h(U).}
\label{collocation-perturbation:eq2}
\end{equation}
\end{itemize}
\begin{remark} Note that for the sake of simplicity we ignore
  quadrature errors and assume that the integrals
  \eqref{collocation-perturbation:eq1} and
  \eqref{collocation-perturbation:eq2} are computed exactly.
\end{remark}
The next step consists in collocating $Q_{h}(u_h(\by))$ with
respect to $\Gamma$.  To this end, we first introduce an auxiliary
probability density function $\hat\rho:\Gamma\rightarrow
\rset^+$ that can be seen as the joint probability of $N$
independent random variables; i.e., it factorizes as
\begin{equation}\label{eq:rhohat}
  \hat\rho(\by) = \prod_{n=1}^{N} \hat\rho_{n}(y_n) \;\; \forall \by\in\Gamma,
  \qquad \text{and is such that }\;\; \left\|\frac{\rho}{\hat
      \rho}\right\|_{L^\infty(\Gamma)}<\infty.
\end{equation}
For each dimension $n=1,\ldots,N$, let $y_{n,k_n}$, $1\leq k_n\leq
p_n+1$, be the $p_n+1$ roots of the orthogonal polynomial $q_{p_n+1}$
with respect to the weight $\hat\rho_{n}$, which then satisfies $
\int_{\Gamma_{n}} q_{p_n+1}(\by) v(\by)
\hat{\rho}_{n}(\by) dy = 0$ for all $v\in
\mcP_{p_n}(\Gamma_{n}).$

Standard choices for $\hat \rho$, such as constant, Gaussian, etc.,
lead to \corr{the well-known} roots of the polynomial $q_{p_n+1}$,
which are tabulated to full accuracy and do not need to be computed.
Note, that for the case of Clenshaw-Curtis abscissas the
  collocation points are chosen as extrema of Chebyshev polynomials.

To any vector of indexes $[k_1, \ldots, k_{N}]$ we associate the global
index
\[
  k=k_1+p_1(k_2-1)+ p_1p_2(k_3-1) +\cdots
\]
and we denote by $y_k$ the point $y_k=[y_{1,k_1}, y_{2,k_2},
\ldots,y_{N,k_{N}}]\in\Gamma$.  We also introduce, for each
$n=1,2,\ldots,N$, the Lagrange basis $\{l_{n,j}\}_{j=1}^{p_n+1}$ of
the space~$\mcP_{p_n}$,
\[
  l_{n,j}\in\mcP_{p_n}(\Gamma_{n}), \qquad
  l_{n,j}(y_{n,k}) = \tilde{\delta}_{jk}, \quad j,k=1,\ldots,p_n+1,
\]
where $\tilde{\delta}_{jk}$ is the Kronecker symbol, and we set $l_k(\by) =
\prod_{n=1}^{N} l_{n,k_n}(y_n)$. Now, let $\mcI_p: C^0(\Gamma)$
$\rightarrow$ $\mcP_p(\Gamma)$, such that
\[
  \mcI_p v(\by) = \sum_{k=1}^{N_p} v(y_k) l_k(\by) \qquad \forall v\in
  C^0(\Gamma).
\]
\noindent Thus for any ${\bf y} \in \Gamma$ we can write the Lagrange
approximation of the QoI ($Q_{h}({\bf y})$):
\[
Q_{h,p}({\bf y}) := {\cal I}_{p} B({\bf y}; u_h({\bf y}),
\varphi_h({\bf y})).
\]
\begin{remark}
For any continuous function $g:\Gamma \rightarrow \rset$ we introduce
the Gauss quadrature formula $\mathbb{E}_{\hat\rho}^p[g]$
approximating the integral $\int_{\Gamma} g({\bf y})\hat\rho(\by)\,
d\by$ as
\begin{equation}
   \mathbb{E}_{\hat\rho}^p[g] = \sum_{k=1}^{N_p} \omega_k g(y_k), \quad
   \omega_k = \prod_{n=1}^N \omega_{k_n}, \quad \omega_{k_n} =
   \int_{\Gamma_{n}} l_{k_n}^2(y)\hat\rho_n(y)\,dy.
\label{eq:gauss}
\end{equation}
In the case $\rho/\hat\rho$ is a smooth function we can use directly
\eqref{eq:gauss} to approximate the mean value or the variance of
$Q_{h}$ as
\begin{align*}
   & \mathbb{E}_h[Q_h] :=
   \mathbb{E}_{\hat\rho}^p\left[\frac{\rho}{\hat\rho}
     Q_{h,p} \right], \,\mbox{and}\,\, \var_h(Q_h):=
   \mathbb{E}_{\hat\rho}^p\left[\frac{\rho}{\hat\rho} Q^2_{h,p} \right]
  -
      \mathbb{E}_{\hat\rho}^p \left[ \frac{\rho}{\hat\rho} 
 Q_{h,p} \right]^{2}.
 \end{align*}
Otherwise, $\mathbb{E}[Q_h]$ and $\var_h(Q_h)$ should be computed with
a suitable quadrature formula that takes into account eventual
discontinuities or singularities of $\rho/\hat\rho$. However, to
simplify the error analysis presentation in Section
\ref{erroranalysis}, we shall assume that the quadrature scheme for
the expectation to be exact.
\label{stochasticcollocation:remark1}
\end{remark}

\subsection{Sparse Grid Approximation}
\label{sparsegrid}

Recall that the dimension of $\mcP_p$ increases as $\prod_{n=1}^N
(p_n+1)$.  This has the consequence that even for a relatively small
dimension $N$ the accurate computation of the mean and variance of
the QoI with a tensor product grid becomes intractable. However, if
the stochastic integral is highly regular with respect to the random
variables, the application of Smolyak sparse grids is well suited.  We
present here a generalization of the classical Smolyak construction
(see e.g.  \cite{Smolyak63,Novak_Ritter_00}) to build a multivariate
polynomial approximation on a sparse grid. See \cite{Back2011} for
details.

Let $\mcI_n^{m(i)}:C^0(\Gamma_{n})\rightarrow {\cal
  P}_{m(i)-1}(\Gamma_{n})$ be the 1D interpolant as previously
introduced. Here $i\geq 1$ denotes the level of approximation and
$m(i)$ the number of collocation points used to build the
interpolation at level $i$, with the requirement that $m(1)=1$ and
$m(i) < m(i+1)$ for $i\geq 1$.  In addition, let $m(0) =0$ and
$\mcI_n^{m(0)}=0$.
Further, we introduce the difference operators
\[
  \Delta_n^{m(i)} := \mcI_n^{m(i)}-\mcI_n^{m(i-1)}.
\]
Given an integer $w \geq 0$ called the approximation level and a
multi-index $\ii=(i_1,\ldots,i_{N})$ $\in \Nset^{N}_+$,
we introduce a function $g:\Nset^{N}_+\rightarrow\Nset$ strictly
increasing in each argument and define a sparse grid
approximation of $Q_h$
\begin{equation}
  \mcS^{m,g}_w[Q_h] 
= \sum_{\ii\in\Nset^{N}_+: g(\ii)\leq w} \;\;
 \bigotimes_{n=1}^{N} \Delta_n^{m(i_n)}(Q_h) 
\end{equation}
\noindent or equivalently written as
\begin{equation}
  \mcS^{m,g}_w[Q_h]
  = \sum_{\ii\in\Nset^{N}_+: g(\ii)\leq w} \;c(\ii)\; 
  \bigotimes_{n=1}^{N} \mcI_n^{m(i_n)}(Q_h), \qquad \text{with } c(\ii) 
  = \sum_{\stackrel{\jj \in \{0,1\}^{N}:}{g(\ii+\jj)\leq w}} (-1)^{|\jj|}.
\label{sparsegrid:eqn1}
\end{equation}

From the previous expression, we see that the sparse grid
approximation is obtained as a linear combination of full tensor
product interpolations.  However, the constraint $g(\ii)\leq w$ in
\eqref{sparsegrid:eqn1} is typically chosen so as to forbid the use of
tensor grids of high degree in all directions at the same time.

Let $\mm(\ii) = (m(i_1),\ldots,m(i_{N}))$ and consider the set of
polynomial multi-degrees
\[
\Lambda^{m,g}(w) = \{\pp\in\Nset^{N}, \;\;  g(\mm^{-1}(\pp+\oone))\leq w\}. 
\]
Denote by $\Pol_{\Lambda^{m,g}(w)}(\Gamma)$ the corresponding
multivariate polynomial space spanned by the monomials with
multi-degree in $\Lambda^{m,g}(w)$, i.e.
\[
\Pol_{\Lambda^{m,g}(w)}(\Gamma) = span\left\{\prod_{n=1}^{N} y_n^{p_n},
  \;\; \text{with } \pp\in\Lambda^{m,g}(w)\right\}.
\]

The following result proved in \cite{Back2011},
states that the sparse approximation formula $\mcS^{m,g}_w$ is exact
in $\Pol_{\Lambda^{m,g}(w)}(\Gamma)$:
\begin{prop}\label{prop1}
\
\begin{itemize} 
\item[a)] For any $f\in C^0(\Gamma;V)$, we have $\mcS_w^{m,g}[f]\in \Pol_{\Lambda^{m,g}(w)}\otimes V$.
\item[b)] Moreover, $\mcS_w^{m,g}[v] = v, \;\; \forall v\in\Pol_{\Lambda^{m,g}(w)}\otimes V$.
\end{itemize}
\end{prop}
\noindent Here $V$ denotes a Banach space defined on $U$ and
\[
C^{0}(\Gamma; V) : = \{ v: \Gamma \rightarrow V\,\, \mbox{is
  continuous on $\Gamma$ and } \max_{y\in \Gamma} \|v(y)\|_{V} < \infty \}.
\]

We recall that the most typical choice of $m$ and $g$ is given by (see
\cite{Smolyak63,Novak_Ritter_00})
\[
m(i) = \begin{cases} 1, & \text{for } i=1 \\ 2^{i-1}+1, & \text{for }
  i>1 \end{cases}\quad \text{ and } \quad g(\ii) = \sum_{n=1}^N
(i_n-1).
\]
This choice of $m$, combined with the choice of Clenshaw-Curtis
interpolation points (extrema of Chebyshev polynomials) leads to
nested sequences of one dimensional interpolation formulas and a
sparse grid with a highly reduced number of points compared to the
corresponding tensor grid.  In Table \ref{sparsegrid:table1} different
choices of $g(\ii)$ are given (see \cite{Back2011}).
\begin{table}[!htbp]
  \begin{tabular}{||c|c|c||}
    \hline\hline
    Approx. space 			&	sparse grid:	 \;\; $m$, $g$				&		polynomial space: \;\; $\Lambda(w)$ \\
    \hline\hline
    Tensor Product 		&	$m(i)=i$ 						&	$\{\pp\in\Nset^{N}: \;\; \max_n p_n \leq w\}$ 	\\
    Product (TP)&	$g(\ii) = \max_n(i_n-1) \leq w$	&											\\	
    \hline
    Total 		    &	$m(i)=i$ 						&	$\{\pp\in\Nset^{N}: \;\; \sum_n p_n \leq w\}$ 	\\
    Degree (TD)   &	$g(\ii) = \sum_n(i_n-1) \leq w$	&											\\
    \hline
    Hyperbolic 	&	$m(i)=i$ 						&	$\{\pp\in\Nset^{N}: \;\; \prod_n(p_n+1) \leq w+1\}$ 	\\
    Cross (HC)&	$g(\ii) = \prod_n(i_n) \leq w+1$	&											\\	
    \hline
    Smolyak (SM) 			&	$m(i)=	\begin{cases} 
      2^{i-1}+1, \, i>1 \\
      1, \, i=1  
    \end{cases}$
    &	$\{\pp\in\Nset^{N}: \;\; \sum_n f(p_n) \leq w\}$ 	\\
    & $g(\ii) = \sum_n(i_n-1) \leq w$ & $ f(p) = \begin{cases}
      0, \; p=0 \\
      1, \; p=1 \\
      \lceil \log_2(p) \rceil, \; p\geq 2
    \end{cases}$				\\
    \hline\hline

	\end{tabular} 
	\caption{Sparse approximation formulas and corresponding set
          of polynomial multi-degrees used for approximation.}
	\label{sparsegrid:table1}
\end{table} 

It is also straightforward to build related anisotropic sparse
approximation formulas by making the function $g$ to act differently
on the input random variables $y_n$. Anisotropic sparse stochastic
collocation \cite{nobile2008b} combines the advantages of isotropic
sparse collocation with those of anisotropic full tensor product
collocation. \cor{Note that in \cite{Chkifa2014}, the authors show
  convergence of sparse grid approximations for en elliptic PDE with
  random coefficients with infinite dimensions i.e. $N = \infty$.}

The mean term $\mathbb{E}[Q_h]$ is approximated as
\begin{equation}
\mathbb{E}[\mcS^{m,g}_{\lv} Q_{h}] = 
\mathbb{E}_{\hat{\rho}}[\mcS^{m,g}_{\lv} Q_{h} \frac{\rho}{\hat{\rho}}],
\label{sparsegrid:eqn2} 
\end{equation}
where $v \in L^{1}_{\rho}(\Gamma)$
\[
\mathbb{E}_{\hat{\rho}}[v] := \int_{\Gamma} v \hat{\rho}(\by)\,\,d\by
\]
and similarly the variance $\var[Q]$ is approximated as
\begin{equation}
\begin{split}
\var_{h}[Q_{h}] 
& = \mathbb{E}[ (\mcS^{m,g}_{\lv} [Q_{h}])^2 ] 
- \mathbb{E}[\mcS^{m,g}_{\lv} [Q_{h}]]^{2} =
\mathbb{E}_{\hat{\rho}}[ (\mcS^{m,g}_{\lv} [Q_{h}])^2 \frac{\rho}{\hat{\rho}}] 
- \mathbb{E}_{\hat{\rho}}[\mcS^{m,g}_{\lv} [Q_{h}] \frac{\rho}{\hat{\rho}}]^{2}.\\
\end{split}
\label{sparsegrid:eqn3}  
\end{equation}

\section{Error Analysis}
\label{erroranalysis}

In this section we derive error estimates of the mean and variance
with respect to (i) the finite element approximation, (ii) the sparse
grid approximation and (iii) truncating the stochastic model to the
first $N_{s}$ dimensions, \corred{again under Assumption
  \ref{analyticity:assumption1} that $a \circ F$ and $\hat g$ do not
  depend on $\omega \in \Omega$, and $f$ is deterministic.}

For notational simplicity we split the Jacobian as follows
\begin{equation}
\partial F(x,\omega) = I + \sum_{l=1}^{N_s} B_{l}(x)
\sqrt{\mu_{l}} Y_{l}(\omega) + \sum_{l=N_s+1}^{N} B_{l}(x)
\sqrt{\mu_{l}} Y_{l}(\omega).
\label{erroranalysis:eqn1}
\end{equation}
Furthermore, let $\Gamma_s := [-1,1]^{N_s}$, $\Gamma_f :=
[-1,1]^{N-N_s}$, then the domain $\Gamma = \Gamma_s \times \Gamma_f$.
We now refer to $Q(\by_s)$ as $Q(\by)$ restricted to the stochastic
domain $\Gamma_{s}$ (i.e. $\mu_l = 0$ for $l = N_s + 1, \dots,
  N$ in eqn \eqref{erroranalysis:eqn1}). A similar notation is used
  for the solution $u(\by_s)$ as the restriction to $\Gamma_s$ of
  $u(\by)$, as well as for $G(\by_s)$. It is clear also that
$Q(\by_s,\by_f)=Q(\by)$ and $G(\by_s,\by_f)=G(\by)$ for all $\by \in
\Gamma_s \times \Gamma_f$, $\by_s \in \Gamma_s$, and $\by_f \in
\Gamma_f$.


Now that we have established notation, we are interested in deriving
estimates for the variance ( $| var[Q({\bf y}_{s},{\bf y}_{f})] -
var[\mcS^{m,g}_w
[Q_{h}({\bf y}_{s})]] | $ ) and mean ($| \eset{Q({\bf
    y}_{s},{\bf y}_{f})}$ $- \eset{\mcS^{m,g}_w[Q_{h}({\bf y}_{s})]}
|$) errors.  First observe that
\[
\begin{split}
|var[Q({\bf y}_{s},{\bf y}_{f})] -
var[\mcS^{m,g}_w[Q_{h}({\bf y}_{s})]] | \leq
& |var[Q({\bf y}_{s},{\bf y}_{f}) ] - var[Q({\bf y}_{s}) ]| \\
& +  |var[Q({\bf y}_{s})] - var[Q_{h}({\bf y}_{s})]| \\ 
& + |var[Q_{h}({\bf y}_{s})] - var[\mcS^{m,g}_w[Q_{h}({\bf y}_{s})]]|.
\end{split}
\]
\noindent Let us analyze the first term. By applying the
Cauchy-Schwartz and Jensen's inequality we have that
\[
\eset{Q({\bf y}_{s},{\bf y}_{f})
^{2} - Q({\bf y}_{s})^{2}}
\leq \|Q({\bf y}_{s},{\bf y}_{f}) - Q({\bf y}_{s}) \|_{L^{2}_{\rho}(\Gamma)} 
\| Q({\bf y}_{s},{\bf y}_{f}) + Q({\bf y}_{s})\|_{L^{2}_{\rho}(\Gamma)},
\]
\corr{or
\[
\eset{Q({\bf y}_{s},{\bf y}_{f})
^{2} - Q({\bf y}_{s})^{2}}
\leq \|Q({\bf y}_{s},{\bf y}_{f}) - Q({\bf y}_{s}) \|_{L^{1}_{\rho}(\Gamma)} 
\| Q({\bf y}_{s},{\bf y}_{f}) + Q({\bf y}_{s})\|_{L^{\infty}_{\rho}(\Gamma)},
\]
}
and 
\[
\begin{split}
|\eset{Q({\bf y}_{s},{\bf y}_{f})
}^{2}
 - \eset{Q({\bf y}_{s})}^{2} |
& = 
|\eset{Q({\bf y}_{s},{\bf y}_{f}) - Q({\bf y}_{s})} \eset{Q({\bf y}_{s}, {\bf y}_f)
 + Q({\bf y}_{s})
} | \\
& \leq 
\|Q({\bf y}_{s},{\bf y}_{f})
 - Q({\bf y}_{s}) \|_{L^{1}_{\rho}(\Gamma)} 
\| Q({\bf y}_{s},{\bf y}_{f}) 
+ Q({\bf y}_{s})\|_{L^{1}_{\rho}(\Gamma)} \\
& \leq 
\|Q({\bf y}_{s},{\bf y}_{f})
 - Q({\bf y}_{s}) \|_{L^{2}_{\rho}(\Gamma)} 
\| Q({\bf y}_{s},{\bf y}_{f}) 
+ Q({\bf y}_{s})\|_{L^{2}_{\rho}(\Gamma)}.
\end{split}
\]
\noindent Therefore
\[
|var[Q({\bf y}_{s},{\bf y}_{f})] - var[Q({\bf y}_{s})]| \leq
C_{T}\|Q({\bf y}_{s},{\bf y}_{f}) - Q({\bf
  y}_{s})\|_{L^{2}_{\rho}(\Gamma)}
\]
\noindent for some positive constant $C_{T} \in
\mathbb{R}^{+}$. It is not hard to show that $|var[Q(\by_s,$
    $\by_f)] - var[\mcS^{m,g}_w[Q_{h}({\bf y}_{s})]] |$ and
  $|\eset{Q(\by_s,\by_f)} - \eset{\mcS^{m,g}_w[Q_{h}(\by_s)}]|$ are
  less or equal to
\[
C_{T} \underbrace{\| Q({\bf y}_{s},{\bf y}_{f})
 - Q({\bf y}_{s})
  \|_{L^{2}_{\rho}(\Gamma)}}_{\mbox{Truncation (I)}} +
\,\,\corr{C_{FE}\underbrace{\|Q({\bf y}_{s}) - Q_{h}({\bf y}_{s})
  \|_{L^{1}_{\rho}(\Gamma_s)}}_{\mbox{Finite Element (II)} }}
\]
\[
+ C_{SG}
\underbrace{
  \| Q_{h}({\bf y}_{s}) - \mcS^{m,g}_w[Q_{h}({\bf y}_{s})] 
\|_{L^{2}_{\rho}(\Gamma_s)}
}
  _{\mbox{Sparse Grid (III)}},
\]
for some positive constants $C_{T},C_{FE}$ and $C_{SG}$. We
now study the error contributions from (I), (II) and (III).

\subsection{Truncation Error (I)}
\label{errorestimates:truncation}
Given that $Q:H^{1}_{0}(U) \rightarrow \mathbb{R}$ is a bounded linear
functional then for any realization of $\varphi(\by_s,\by_f)$ we have
that
\[
\begin{split}
|Q({\bf y}_s,{\bf y}_f) - Q({\bf y}_s)|
& =
|B(\by_s,\by_f; \varphi(\by_s,\by_f), \tilde{u}({\bf y}_s,{\bf y}_f) - \tilde{u}({\bf
  y}_s))| \\
& \leq a_{max} \F^{d}_{max}\F^{-2}_{min} 
\|\varphi(\by_s,\by_f)\|_{H^{1}_{0}(U)} \| \tilde{u}({\bf y}_s,{\bf y}_f) - \tilde{u}({\bf
  y}_s) \|_{H^{1}_{0}(U)},
\end{split}
\]
\corg{where we have denoted by $\tilde u(\by_s) = (\tilde u \circ
  F)(\cdot,\by_s)$ and similarly $\tilde u(\by_s,\by_f) = (\tilde u
  \circ F)(\cdot,\by_s,\by_f)$.}
Following a similar argument as the proof from Lemma
\ref{setup:lemma1} we have that
\[
\|\varphi(\by_s,\by_f)\|_{H^{1}_{0}(U)} \leq 
\frac{\corg{\| q  \|_{L^{2}(\tilde D)}} C_{P}(U) \F^{d+2}_{max}  
}{a_{min} \F^{d}_{min}}
\]
a.s., where $q$ is defined in section \ref{setup:QoI}. Thus
\[
\|Q({\bf y}_s,{\bf y}_f) - Q({\bf y}_s)\|_{L^{2}_{\rho}(\Gamma)}
\leq C_{TR}
\| \tilde{u}({\bf y}_s,{\bf y}_f) - \tilde{u}({\bf
  y}_s) \|_{L^{2}(\Gamma; H^{1}_{0}(U))},
\]
where \corg{$C_{TR} := a_{max}a_{min}^{-1}
  \F^{2d+2}_{max}\F^{-d-2}_{min} \| q \|_{L^{2}(\tilde U)}
  C_{P}(U)$}. We now seek control on the error term $e:= \|
\tilde{u}({\bf y}_s,{\bf y}_f) - \tilde{u}({\bf y}_s)
\|_{L^{2}_{\rho}(\Gamma; H^{1}_{0}(U))}$.  First we establish some
notation and definitions.  From Section \ref{analyticity} we have
shown that the solution $\tilde{u} \circ F$ of Problem
\ref{setup:Prob3} varies continuously with respect to $ y \in
\Gamma$. More precisely, recall that if $V$ is a Banach space defined
on $U$ and
\[
C^{0}(\Gamma; V) : = \{ v: \Gamma \rightarrow V\,\, \mbox{is
   continuous on $\Gamma$ and } \max_{y\in \Gamma} \|v(y)\|_{V} < \infty 
 \},
\]
\noindent then $u \in C^{0}(\Gamma, H^{1}_{0}(U))$. Furthermore, let
\[
L^{2}_{\rho}(\Gamma;V) := \{v:\Gamma \rightarrow V\,\, \mbox{is strongly
  measurable and} \,\, \int_{\Gamma} \|v\|^{2}_{V}\,\rho(\by)\,d\by <
\infty \}.
\]
\corr{From Theorem \ref{analyticity:theorem1} we have that $\tilde{u} \circ F \in
C^{0}(\Gamma;H^{1}_{0}(U)) \subset L^{2}_{\rho}(\Gamma;H^{1}_{0}(U))$, thus
$\tilde{u} \circ F$ satisfies the following variational problem}
\[
{\cal A}(\tilde{u} \circ F,v) : = E[B({\bf y}_s,{\bf y}_f
; \tilde{u} \circ F,v)] = E[\tilde{l}(\by_s,\by_f;v)]\,\,\, \forall v \in
L^{2}_{\rho}(\Gamma;H^{1}_{0}(U)).
\]
The following lemma will be useful in deriving error estimates.
\begin{lem} For all $w,v \in L^{2}_{\rho}(\Gamma;H^{1}_{0}(U))$ we
have that
\[
|{\cal A}(w,v)| \leq a_{max} \F^{d}_{max}\F^{-2}_{min}
\|w\|_{L^{2}_{\rho}(\Gamma; H^{1}_{0}(U))} \|v\|_{L^{2}_{\rho}(\Gamma;
  H^{1}_{0}(U))}.
\]
\label{errorestimates:truncation:lemma1}
\end{lem}
\begin{pf}
\[
\begin{split}
|{\cal A}(w,v)| & \leq \sup_{\by \in \Gamma} \lambda_{max}(G(\by)) 
\mathbb{E} \left[\int_{U} \corg{| \nabla
  w| |\nabla v |} \right] \\ 
& \leq a_{max} \F^{d}_{max}\F^{-2}_{min}
 \mathbb{E}[\|\nabla w\|_{L^{2}(U)}
\|\nabla v\|_{L^{2}(U)}
] \\
& \leq a_{max} \F^{d}_{max}\F^{-2}_{min}
 \|w\|_{L^{2}_{\rho}(\Gamma; H^{1}_{0}(U))}
\|v\|_{L^{2}_{\rho}(\Gamma; H^{1}_{0}(U))}.
\\
\end{split}
\]
\end{pf}
We can now derive the truncation error (I).


\begin{thm}
 Let $\tilde{u}$ be the solution to the linear Problem
 \ref{setup:Prob3} that satisfies Assumptions \ref{setup:Assumption1},
 \ref{setup:Assumption2}, \ref{setup:Assumption3},
 \ref{setup:Assumption4} and \ref{analyticity:assumption1}.
 \corred{Furthermore, assume that $\|(f \circ
   F)(\by)\|_{W^{1,\infty}(U)}$ is \corred{bounded uniformly in
     $\Gamma$} then} 
\[
\begin{split}
  \|\tilde{u}({\bf y}_s,{\bf y}_f) - 
 \tilde{u}({\bf y}_s)\|_{L^{2}_{\rho}(\Gamma; H^{1}_{0}(U))} 
&\leq \C_{1} B_{\T} + \C_2 C_{\T},
\end{split}
\]
where $B_{\T} := \sup_{x \in U} \sum_{i = N_s+1}^{N} \sqrt{\mu_{i}}
\|B_{i}(x)\|$, $C_{\T} : = \sum_{i = N_s+1}^{N} \sqrt{\mu_{i}}
\|b_{i}(x)\|_{L^{\infty}(U)}$,
\[
\begin{split}
\C_1
& := {\cal C} 
\Big(C_{P}(U) \F_{max}^{d-1} \F_{min}^{-2}d
\|f \circ F(\by_s,\by_f) \|_{L^{2}(\Gamma;L^{2}(U))} \\
&+ 
\corred{\corm{a_{max}} H(\F_{max},\F_{min},\tilde{\delta},d) 
\| \hat \bw  \|_{L^{2}(\Gamma;H^{1}_0(U))}} \\
&+
\corred{a_{max} H(\F_{max},\F_{min},\tilde{\delta},d)
\| \tilde{u}(\by_s,\by_f) \|_{L^{2}_{\rho}(\Gamma; H^{1}_{0}(U))}}
\Big), \\
\C_2          
&:=
\corred{{\cal C} \| \hat v \|_{L^{\infty}(U)}
\F_{max}^{d} C_P(U) 
\| f \|_{L^{\infty}(\Gamma;W^{1,\infty}(U))}
\| \chi_{U} \|_{L^{2}(U)} }, \\
\end{split}
\]
\corred{${\cal G}:= \cup_{\omega \in \Omega} D(\omega)$}, ${\cal C} :=
\frac{C_{P}(U)^{2}}{a_{min}\F_{min}^{d} \F_{max}^{-2}}$, and
$H(\F_{max},\F_{min},\tilde{\delta},d):=
\F^{d-1}_{max}\F_{min}^{-3}(\F_{max}(2 + \F_{min}^{-1}(1 -
\tilde{\delta})) + \F^{-1}_{min}d)$.
\end{thm}
\begin{pf}
We follow a similar strategy as in \cite{Babuska2004,Frauenfelder2005}
to compute the bounds for the truncation of the stochastic variables
to $\Gamma_{s}$. Consider the solution to Problem \ref{setup:Prob3}
$\tilde{u}(\by_s) \in C^{0}(\Gamma_s;H^{1}_{0}(U)) \subset
L^{2}_{\rho}(\Gamma; H^{1}_{0}(U))$, where the matrix of coefficients
$G(\by_s)$ depends only on the variables $Y_{1}, \dots, Y_{N_{s}}$,
\[
\mathbb{E}[B({\bf y}_{s}; \tilde{u}(\by_s),v)] =
  \mathbb{E}[\tilde{l}(\by_s;v)]\,\,\, \forall v \in
  L^{2}_{\rho}(\Gamma_s; H^{1}_{0}(U)).
\]
Furthermore the variational form is still valid $\forall v \in
L^{2}_{\rho}(\Gamma; H^{1}_{0}(U))$ i.e.
\[
{\cal A}_{N_s}(\tilde{u}(\by_s),v)
 := 
\eset{B({\bf y}_{s}; \tilde{u}(\by_s),v)} = \eset{\tilde{l}(\by_s;v)}\,\,\, 
\forall v \in L^{2}_{\rho}(\Gamma; H^{1}_{0}(U)).
\]
Now, Observe that $\forall v \in
L^{2}_{\rho}(\Gamma;H^{1}_{0}(U))$ we have that
\[
\begin{split}
{\cal A}_{N_s}(v,v) & \geq \inf_{\by_s \in \Gamma_s} 
              \lambda_{min}(G(\by_s)) \mathbb{E}[\|\nabla v\|^{2}_{L^{2}(U)}]\\
            & \geq a_{min} \F^{d}_{min}\F^{-2}_{max} 
            \mathbb{E}[\|\nabla v\|^{2}_{L^{2}(U)}] \\
            & \geq a_{min} \F^{d}_{min}\F^{-2}_{max} C_P(U)^{-2}
         \corg{\|v\|^{2}_{L^{2}_{\rho}(\Gamma; H^{1}_{0}(U))}}.
\end{split}
\]
By adapting the proof from Strang's Lemma and applying Lemma
\ref{errorestimates:truncation:lemma1} we have that for all $v \in
L^{2}_{\rho}(\Gamma;H^{1}_{0}(U))$
\[
\begin{split}
  \| \tilde{u}(\by_s) - v \|^{2}_{L^{2}_{\rho}(\Gamma; H^{1}_{0}(U))}
  & \leq \corg{\cal C} (
 {\cal A}_{N_s}
(\tilde{u}(\by_s)  - v, \tilde{u}(\by_s)  - v) 
\pm {\cal A}(\corg{\tilde{u}(\by)} - v, \tilde{u}(\by_s) - v) ) \\
  & 
\leq \corg{\cal C} (a_{max} \F^{d}_{max}\F^{-2}_{min} 
 \| \corg{\tilde{u}(\by)} - v \|_{L^{2}_{\rho}(\Gamma; H^{1}_{0}(U))} 
\| \tilde{u}(\by_s) - v
  \|_{L^{2}_{\rho}(\Gamma; H^{1}_{0}(U))} \\
&+ 
| {\cal A}(v, \tilde{u}(\by_s) - v) -
   {\cal A}_{N_s}(v, \tilde{u}(\by_s) - v)| \\
&\corr{+|\eset{\tilde{l}(\by_s;\tilde{u}(\by_s) - v) - 
\tilde{l}(\by_s,\by_f;\tilde{u}(\by_s) - v)}|}).
\end{split}
\]
\corg{Now, pick $v = \tilde{u}(\by_s,\by_f)$, thus
\[
  \| \tilde{u}(\by_s,\by_f) - \tilde{u}(\by_s) \|_{L^{2}_{\rho}(\Gamma; H^{1}_{0}(U))}
 \leq \corg{\cal C}({\cal B}_1 + {\cal B}_2
),
\]
where
\[
\begin{split}
{\cal B}_1 & := \sup_{w \in L^{2}_{\rho}(\Gamma; H^{1}_{0}(U))}
\frac{|{\cal A}(\tilde u(\by_s,\by_f), w) - {\cal A}_{N_s}(\tilde u(\by_s,\by_f),
  w)|}{\| w \|_{L^{2}_{\rho}(\Gamma; H^{1}_{0}(U))}} \\ 
& \leq 
\| \tilde{u}(\by_s,\by_f) \|_{L^{2}_{\rho}(\Gamma; H^{1}_{0}(U))} 
\sup_{x
  \in U, {\bf y} \in \Gamma} 
\|G({\bf y}_s,{\bf y}_f) -
G({\bf y}_{s})\|,
\end{split}
\]
and
\[
\corg{{\cal B}_2 := \sup_{w \in L^{2}_{\rho}(\Gamma;
    H^{1}_{0}(U))} \| w \|^{-1}_{L^{2}_{\rho}(\Gamma; H^{1}_{0}(U))}
|\eset{\tilde{l}(\by_s;\tilde{u}(\by_s) - w) - 
\tilde{l}(\by_s,\by_f;\tilde{u}(\by_s) - w)}|.}
\]
}
\noindent {\bf Bound for ${\cal B}_1$:}
For notational simplicity we rewrite \eqref{analyticity:eqn1} as
\[
\partial F({\bf y}_s,{\bf y}_f) = I + A^{s}_{N_{s}}({\bf y}_{s})
 + A^{f}_{N_{f}}({\bf y}_{f})
\]
for some set of matrices $A^{s}_{N_{s}}, A^{f}_{N_{f}} \in \R^{d
  \times d} \times U \times \Gamma $.  With a slight abuse of notation
we refer to $\partial F({\bf y}_s)$ as $\partial F({\bf y}_s) := I +
A^{s}_{N_{s}}({\bf y}_{s})$. Note that $\F_{min} \leq
\sigma_{min}(\partial F({\bf y}_s,{\bf y}_f)) \Rightarrow \F_{min}
\leq \sigma_{min}(\partial F({\bf y}_s))$ and $\sigma_{min}(\partial
F({\bf y}_s,{\bf y}_f)) \leq \F_{max} \Rightarrow
\sigma_{min}(\partial F({\bf y}_s))\leq \F_{max}$.

We now estimate the term $\|G({\bf y}_{s},{\bf y}_{f}) - G({\bf
    y}_{s})\|$.  Denoting $J({\bf y}_{s},{\bf y}_{f}):=|\partial
  F({\bf y}_s,{\bf y}_f)|$ $\partial F({\bf y}_{s},\by_f)^{-1}
  \partial F({\bf y}_{s},\by_f)^{-T}$ and similarly for $J({\bf
    y}_{s})$ we have
\begin{equation}
\corg{
\begin{split}
\|G({\bf y}_{s},{\bf y}_{f}) - G({\bf y}_{s})\|
& \leq 
a_{max}\| J({\bf y}_{s},{\bf y}_{f}) -    J(\by_s) \|.
\end{split} 
}
\label{errorestimates:truncation:eqn1}
\end{equation}
Now, \corred{$\forall x\in U$ and $\forall \by \in \Gamma$}
\begin{equation}
\begin{split}
\| J({\bf y}_{s},{\bf y}_{f}) -    J(\by_s) \| 
&\corr{=
\| J({\bf y}_{s},{\bf y}_{f}) -    J(\by_s) 
\pm |\partial F({\bf y}_s,{\bf y}_f)| \partial F({\bf y}_{s})^{-1} 
\partial F({\bf y}_{s})^{-T} \| }\\
& \leq 
|\partial F({\bf y}_s,{\bf y}_f)|\| 
\partial F({\bf y}_{s},\by_f)^{-1} 
\partial F({\bf y}_{s},\by_f)^{-T}
-
\partial F({\bf y}_{s})^{-1} 
\partial F({\bf y}_{s})^{-T}
\| \\
&+ 
\big||\partial F({\bf y}_s,{\bf y}_f)| - |\partial F({\bf y}_s)|\big|
\|\partial F({\bf y}_{s})^{-1} 
\partial F({\bf y}_{s})^{-T}\| \\
&\leq 
\F_{max}^{d}\| 
\partial F({\bf y}_{s},\by_f)^{-1} 
\partial F({\bf y}_{s},\by_f)^{-T}
-
\partial F({\bf y}_{s})^{-1} 
\partial F({\bf y}_{s})^{-T}
\| \\
& + \F_{min}^{-2} \big|
|\partial F({\bf y}_s,{\bf y}_f)| - |\partial F({\bf y}_s)| \big|.
\end{split}
\label{errorestimates:truncation:eqn2}
\end{equation}
Applying the matrix identity $(A-BD^{-1}C)^{-1} =
A^{-1}+A^{-1}B(D-CA^{-1}B)^{-1}CA^{-1}\,$ where $A = I +
A^{s}_{N_{s}}({\bf y}_{s})$, $B = -A^{f}_{N_{f}}({\bf y}_{f})$ and $C =
D = I$ we obtain
\[
\partial F({\bf y}_{s},\by_f)^{-1} = 
\partial F({\bf y}_{s})^{-1} 
+ E({\bf y}_{s},\by_f),
\]
where
\[
\begin{split}
E({\bf y}_{s},\by_f) &:=-\partial F({\bf y}_{s})^{-1}
A^{f}_{N_{f}}({\bf y}_{f})(I + \partial F({\bf y}_{s})^{-1}
A^{f}_{N_{f}}({\bf y}_{f}))^{-1} \partial F({\bf y}_{s})^{-1} \\ & =
-\partial F({\bf y}_{s})^{-1} A^{f}_{N_{f}}({\bf y}_{f})
\partial F({\bf y}_{s},\by_f)^{-1}, 
\end{split}
\]
then
\[
\begin{split}
\partial F({\bf y}_{s},\by_f)^{-1} 
\partial F({\bf y}_{s},\by_f)^{-T}
-
\partial F({\bf y}_{s})^{-1} 
\partial F({\bf y}_{s})^{-T}
&=  E({\bf y}_{s},\by_f) E({\bf y}_{s},\by_f)^{T} \\
&+ \partial F({\bf y}_{s})^{-1}E({\bf y}_{s},\by_f)^{T} \\
&+ E({\bf y}_{s},\by_f)\partial F({\bf y}_{s})^{-T}.
\end{split}
\]
Now, \corred{$\forall x\in U$ and $\forall \by \in \Gamma$}
\[
\| E({\bf y}_{s},\by_f)\| 
\leq \F_{min}^{-2} 
\sum_{i = N_s+1}^{N} \sqrt{\mu_{i}} \|B_{i}(x)\|.
\]
It follows that \corred{$\forall x\in U$ and $\forall \by \in \Gamma$}
\begin{equation}
\begin{split}
\| 
\partial F({\bf y}_{s},\by_f)^{-1} 
\partial F({\bf y}_{s},\by_f)^{-T}
-
\partial F({\bf y}_{s})^{-1} 
\partial F({\bf y}_{s})^{-T}
\| & \leq 
B_{\T}\F_{min}^{-3}(2 + \F_{min}^{-1} (1 - \tilde{\delta})). \\
\end{split}
\label{errorestimates:truncation:eqn3}
\end{equation}
From Theorem 2.12 in \cite{Ipsen2008}
($A,E \in \C^{d \times d}$ then $|det(A + E) - det(A)| \leq d \|E\| \max \{
\|A\|,$ $\|A + E\|\}^{d-1})$ we obtain
\corred{$\forall x\in U$ and $\forall \by \in \Gamma$}
\begin{equation}
\big| |\partial F({\bf y}_s,{\bf y}_f)| -|\partial F({\bf y}_{s})|
\big| \leq \F_{max}^{d-1} \F^{-2}_{min}B_{\T}d.
\label{errorestimates:truncation:eqn4}
\end{equation}
Combining \corr{equations} 
\eqref{errorestimates:truncation:eqn2},
\eqref{errorestimates:truncation:eqn3} and
\eqref{errorestimates:truncation:eqn4} we obtain
\[
\corg{\sup_{x \in U, \by \in \Gamma}} \|J({\bf y}_s,{\bf y}_f) -
J({\bf y}_{s}) \| \leq  B_{\T}
H(\F_{max},\F_{min},\tilde{\delta},d).
\]
\corg{Thus,
\[
\begin{split}
{\cal B}_1
 & \leq 
a_{max} B_{\T}H(\F_{max},\F_{min},\tilde{\delta},d)
\| \tilde{u}(\by_s,\by_f) \|_{L^{2}_{\rho}(\Gamma; H^{1}_{0}(U))}.
\end{split}
\]
}
\noindent {\bf Bound for ${\cal B}_2$:}
\begin{equation}
\begin{split}
{\cal B}_2 
&\leq  
\sup_{w \in L^2_{\rho}(\Gamma;H^1_0(U))} 
\|w\|^{-1}_{L^2_{\rho}(\Gamma;H^1_0(U))}
| \mathbb{E}\Big[
\int_U (f \circ F(\by_s) |\partial F(\by_s)|  \\
&- 
f \circ F(\by_s,\by_f) |\partial F(\by_s,\by_f)|)
w \\
&+      
\int_U \corm{(\nabla \hat \bw)^T(
G(\by_s) -
G(\by_s,\by_f))}
\nabla w \Big] | \\
& \leq 
\sup_{w \in L^2_{\rho}(\Gamma;H^1_0(U))} 
\|w\|^{-1}_{L^2_{\rho}(\Gamma;H^1_0(U))}
\mathbb{E}\Big[
\int_U 
|(f \circ F(\by_s)  - f \circ F(\by_s,\by_f))
|\partial F(\by_s)|w| 
\\
&+     
|f \circ F(\by_s,\by_f)(|\partial F(\by_s)| 
- |\partial F(\by_s,\by_f)|)w| \\
&
+ 
\corm{
|(\nabla \hat \bw )^T (G(\by_s) - G(\by_s,\by_f)) \nabla w |
}
\Big].
\end{split}
\label{erroranalysis:eqn2}
\end{equation}
Now we have that
\begin{equation}
\begin{split}
& 
\sup_{w \in L^2_{\rho}(\Gamma;H^1_0(U))} 
\|w\|^{-1}_{L^2_{\rho}(\Gamma;H^1_0(U))}
\mathbb{E} \Big[
\int_U 
|(f \circ F(\by_s)  - f \circ F(\by_s,\by_f))
|\partial F(\by_s)| w| \Big] \\
& \leq 
\sup_{w \in L^2_{\rho}(\Gamma;H^1_0(U))} 
\|w\|^{-1}_{L^2_{\rho}(\Gamma;H^1_0(U))}
\mathbb{E}\Big[ \corred{\max_{x \in U}}
|\partial F(\by_s)|
\corred{
\| f \circ F(\by_s)  - f \circ F(\by_s,\by_f) \|_{L^{2}(U)}}
\| w \|_{L^{2}(U)} \Big] \\
& \leq 
\sup_{w \in L^2_{\rho}(\Gamma;H^1_0(U))} 
\|w\|^{-1}_{L^2_{\rho}(\Gamma;H^1_0(U))}
\F_{max}^{d}
\mathbb{E}\Big[ 
\| \corred{f \circ F(\by_s)  - f \circ F(\by_s,\by_f)} \|_{L^{2}(U)} \\
&
C_P(U)
\| w \|_{H^{1}_0(U)} \Big] \\
& \leq 
\F_{max}^{d} C_P(U)
\mathbb{E}\Big[ 
\| \corred{f(\cdot,\by)} \|_{W^{1,\infty}(\corred{\cal G})}
\|\corred{F(\by_s)  - F(\by_s,\by_f)}\|_{L^{2}(U)} \Big] \\
& \leq 
\F_{max}^{d} C_P(U) \| \corred{f} \|_{L^{\infty}(\Gamma;W^{1,\infty}(\corred{\cal G}))}
\| \chi_{U} \|_{L^{2}(U)}
\sup_{\by \in \Gamma, x \in \Gamma} |F(\by_s)  - F(\by_s,\by_f) |,
\end{split}
\label{erroranalysis:eqn3}
\end{equation}
where $\chi_{U}$ is the characteristic function defined on $U$. Note
that \corred{$\forall x\in U$ and $\forall \by \in \Gamma$}
\begin{equation}
\begin{split}
&     
\sup_{w \in L^2_{\rho}(\Gamma;H^1_0(U))} 
\|w\|^{-1}_{L^2_{\rho}(\Gamma;H^1_0(U))}
\E \Big[ 
\int_U |f \circ F(\by_s,\by_f)(|\partial F(\by_s)| 
- |\partial F(\by_s,\by_f)|)w| \Big] \\
&\leq
\sup_{w \in L^2_{\rho}(\Gamma;H^1_0(U))} 
\|w\|^{-1}_{L^2_{\rho}(\Gamma;H^1_0(U))}
\E \Big[ \int_U \|f \circ F(\by_s,\by_f) \|_{L^{2}(U)} \|w\|_{L^{2}(U)} \Big] \\
&
\sup_{\by \in \Gamma, x \in U } ||\partial F(\by_s)| 
- |\partial F(\by_s,\by_f)||
\\
&\leq
C_P(U)
\|f \circ F(\by_s,\by_f) \|_{L^{2}(\Gamma;L^{2}(U))} 
\sup_{\by \in \Gamma, x \in U } ||\partial F(\by_s)| 
- |\partial F(\by_s,\by_f)||, 
\end{split}
\label{erroranalysis:eqn4}
\end{equation}
and
\begin{equation}
\begin{split}
&\sup_{w \in L^2_{\rho}(\Gamma;H^1_0(U))} 
\|w\|^{-1}_{L^2_{\rho}(\Gamma;H^1_0(U)}
\E \Big[ 
\corm{(\nabla \hat \bw)^T
(G(\by_s)
-
G(\by_s,\by_f))}
\nabla w 
\Big] 
\\
&\leq
\| \corm{\hat \bw \|_{L^{2}(\Gamma;H^{1}_0(U))}
\sup_{\by \in \Gamma, x \in U} \|G(\by) - G(\by_s)\|}.
\end{split}
\label{erroranalysis:eqn6}
\end{equation}
Substituting \eqref{errorestimates:truncation:eqn1},
\eqref{errorestimates:truncation:eqn4}, \eqref{erroranalysis:eqn3},
\eqref{erroranalysis:eqn4}, 
and
\eqref{erroranalysis:eqn6} in \eqref{erroranalysis:eqn2} we obtain
\[
\begin{split}
{\cal B}_2
&\leq 
B_{\T}\Big(C_{P}(U) \F_{max}^{d-1} \F_{min}^{-2}d
\|f \circ F(\by_s,\by_f) \|_{L^{2}(\Gamma;L^{2}(U))} \\
&+ 
\corm{a_{max}} H(\F_{max},\F_{min},\tilde{\delta},d) 
\| \corred{\hat \bw} \|_{L^{2}(\Gamma;H^{1}_0(U))} \Big) \\
&+
\corred{C_{\T} \| \hat v \|_{L^{\infty}(U)}
\F_{max}^{d} C_P(U) \| f \|_{L^{\infty}(\Gamma;W^{1,\infty}({\cal G}))}
\| \chi_{U} \|_{L^{2}(U)} }.
\end{split}
\]
The result follows.
\end{pf}
\subsection{Finite Element Error (II)} 
\label{erroranalysis:finiteelement}
 
The second quantity controls the convergence with respect to the mesh
size $h$. This will be determined by the polynomial order of the
finite element subspace $H_{h}(U) \subset H^{1}_{0}(U)$ and the
regularity of the solution $u$.  From
\eqref{collocation:perturbation:finiteelement} we obtain the following
bound:
\[
\|\tilde{u}({\bf y}_{s}) - u_{h}({\bf y}_{s})
\|_{L^{2}_{\rho}(\Gamma_s; H^{1}_{0}(U))} \leq
C_{\Gamma_s}(r)h^{r}
\]
for some constant $r \in \N$ and $C_{\Gamma_s}(r):= \int_{\Gamma_s}
C(r,\tilde{u}(\by_s))\rho(\by_s)d\by_s$. The constant $r$ depends on the
polynomial degree of the finite element basis and the regularity
properties of the solution $u$ (which is dependent on the regularity
of $f$, the diffusion coefficient $a$ and the mapping $F$). Similarly
the error for the influence function is characterized as
\[
\|\varphi({\bf y}_{s}) - \varphi_{h}({\bf y}_{s})
\|_{L^{2}_{\rho}(\Gamma_s; H^{1}_{0}(U))} \leq D_{\Gamma_s}(r)h^{r},
\]
where $D_{\Gamma_s}(r):= \int_{\Gamma_s}
C(r,\varphi(\by_s))\rho(\by)d\by$. Following duality arguments we
obtain
\begin{equation}
\eset{|Q(\tilde{u}({\bf y}_{s}) - u_{h}({\bf y}_{s})) | }
 \leq 
a_{max} \F^{d}_{max} \F^{-2}_{min} C_{\Gamma_s}(r)D_{\Gamma_s}(r)
h^{2r}.
\label{errorestimates:finiteelement:eqn1}
\end{equation}
\noindent 
\subsection{Sparse Grid Error (III)} 
In this section we shall not enumerate all the convergence rates that
depend on the formulas from Table \ref{sparsegrid:table1}, but refer
the reader to the appropriate citations. However, we will only
explicitly derive the convergence rates for the isotropic Smolyak
sparse grid.

Given the bounded linear functional $Q$ we have that
 \[
   \| Q(u_h({\bf y}_{s})) - Q(\mcS^{m,g}_w[u_h({\bf y}_{s})])
   \|_{L^{2}_{\rho}(\Gamma_s)} \leq a_{max}\F^{d}_{max} \F_{min}^{-2}
   \|e\|_{L^{2}_{\rho}(\Gamma_{s};H^{1}_{0}(U))},
\]
where $e:=u_h({\bf y}_{s}) - \mcS^{m,g}_w[u_h({\bf y}_{s})]$.
However, as noted in Section \ref{sparsegrid}, the sparse grid is
computed with respect to the auxiliary density function $\hat{\rho}$,
thus
\[
\|e\|_{L^{2}_{\rho}(\Gamma_{s};H^{1}_{0}(U))} \leq
 \left\|
\frac{\rho}{\hat{\rho}}
\right\|_{L^{\infty}(\Gamma_s)}
\|e\|_{L^{2}_{\hat{\rho}}(\Gamma_{s};H^{1}_{0}(U))}.
\]
The error term
$\|e\|_{L^{2}_{\hat{\rho}}(\Gamma_{s};H^{1}_{0}(U))}$ is
controlled by the number of collocation knots $\eta$ (or work), the
choice of the approximation formulas $(m(i), g({\bf i}))$ from Table
\ref{sparsegrid:table1}, and the region of analyticity of
$\Theta_{\beta} \subset \mathbb{C}^{N_s}$. From Theorem
\ref{analyticity:theorem1} the solution $\tilde{u}(\by_s)$ admits an
extension in $\C^{N_s}$ i.e. $\by_s \rightarrow \bz_s \in \C^{N_s}$
and $\tilde{u}(\bz_{s}) \in C^{0}(\Theta_{\beta};H^{1}_{0}(U))$. All
the results proved in Section \ref{analyticity} can be obtained also
for the semi-discrete solution $u_{h}({\by_s})$ which admits an
analytic extension in the same region $\Theta_{\beta}$ and
$u_{h}(\bz_s) \in C^{0}(\Theta_{\beta};H_{h}(U))$.

In \cite{nobile2008b, nobile2008a} the authors derive error estimates
for isotropic and anisotropic Smolyak sparse grids with
Clenshaw-Curtis and Gaussian abscissas where
$\|e\|_{L^{2}_{\hat{\rho}}(\Gamma_{s};H^{1}_{0}(U))}$ exhibit
algebraic or subexponential convergence with respect to the number of
collocation knots $\eta$ (See Theorems 3.10, 3.11, 3.18 and 3.19
\corg{in \cite{nobile2008a}} for more details). However, for these
estimates to be valid the solution $u$ has to admit and extension on a
polyellipse in $\C^{N_s}$, ${\cal E}_{\sigma_1, \dots, \sigma_{N_s}} :
= \Pi_{i=1}^{N_s}{\cal E}_{n,\sigma_n}$, where
\[
{\cal E}_{n,\sigma_n} = \left\{ z \in \C;\,\Real(z) = \frac{e^{\sigma_n}
 +
  e^{-\sigma_n}}{2}cos(\theta),\,\,\,\Imag(z)= \frac{e^{\sigma_n} -
  e^{-\sigma_n}}{2}sin(\theta),\theta \in [0,2\pi) \right\},\,\,\,
\]
and $\sigma_n > 0$.  
For an isotropic sparse grid the overall asymptotic subexponential
decay rate $\hat{\sigma}$ will be dominated by the smallest $\sigma_n$
i.e.
\[
\hat{\sigma} \equiv \min_{n = 1,\dots, N_s} \sigma_n.
\]
Then the goal is to choose the largest $\hat{\sigma}$ such that ${\cal
  E}_{\sigma_1, \dots, \sigma_{N_s}} \subset \Theta_{\beta}$. First,
recall from Section \ref{analyticity} that
\[
\corr{\Theta_{\beta} : = \left\{ {\bf z} \in \mathbb{C}^{N};\, {\bf z} =
      {\bf y} + {\bf w},\,{\bf y} \in [-1,1]^{N_s},\, \sum_{l=1}^{N_s}
      \sup_{x \in U} \| B_l(x) \|_{2} \sqrt{\mu_{l}}
      |w_{l}| \leq \beta  \right\}}.
\]
We can now form the set $\Sigma \subset \C^{N_s}$ such that
$\Sigma \subset \Theta_{\beta}$, where $\Sigma:=\Sigma_1 \times \dots
\times \Sigma_{N_s}$ and
\[
\Sigma_{n} := \left\{ {\bf z} \in \mathbb{C};\, {\bf z} = {\bf y}
+ {\bf w},\,{\bf y} \in [-1,1],\, |w_n| \leq \tau_n :=
\frac{\beta}{1 - \tilde{\delta}}
\right\}.
\]
for $n = 1,\dots,N_s$.  The polyellipse ${\cal E}_{\sigma_1, \dots,
  \sigma_n}$ can now be embedded in $\Sigma$ by choosing
  $\sigma_1 = \sigma_2 = \dots = \sigma_{N_s} = \hat{\sigma} =
  \log{(\sqrt{\tau^2_{N_s} + 1} + \tau_{N_s} )} > 0$.


From Theorem 3.11 \cite{nobile2008a}, given a sufficiently large
$\eta$ for a nested CC sparse grid we obtain the following estimate
\begin{equation}
\|e\|_{L^{2}_{\hat{\rho}}(\Gamma_{s};H^{1}_{0}(U))} \leq {\cal
  Q}(\sigma,\delta^{*},N_s)\eta^{\mu_3(\sigma,\delta^{*},N_s)}\exp
  \left(-\frac{N_s \sigma}{2^{1/N_s}} \eta^{\mu_2(N_s)} \right), \\
\label{erroranalysis:sparsegrid:estimate}
\end{equation}
where 
\[{\cal Q}(\sigma,\delta^{*},N_s) := 
\frac{C_1(\sigma,\delta^{*})}{\exp(\sigma \delta^{*} \tilde{C}_2(\sigma)  )}
\frac{\max\{1,C_1(\sigma,\delta^{*})\}^{N_s}}{|1 - C_1(\sigma,\delta^{*})|}, 
\]
$\sigma = \hat{\sigma}/2$, $\mu_2(N_s) = \frac{log(2)}{N_s(1 +
  log(2N_s))}$ and $\mu_3(\sigma,\delta^{*},N_s,) = \frac{\sigma
  \delta^{*}\tilde{C}_2(\sigma)}{1 + \log{(2N_s)}}$.  The constants
$C_1(\sigma,\delta^{*})$, $\tilde{C}_2(\sigma)$ and $\delta^{*}$ are
defined in \cite{nobile2008a} \corr{equations} (3.11) and (3.12).

\section{Complexity and Tolerance}
\label{complexity} In this section we derive
the total work $W$ needed such that $|var[Q(\by_s,$ $\by_f)] -
var[\mcS^{m,g}_w[Q_{h}({\bf y}_{s})]] |$ and
$|\mathbb{E}[Q(\by_s,\by_f)]$ $-\mathbb{E}[\mcS^{m,g}_w$ $[
    Q_{h}(\by_s)]]|$ for the isotropic CC sparse grid are less than or equal
to a given tolerance parameter $tol \in \R^{+}$.

Let $N_{h}$ be the number of degrees of freedom to solve the
semi-discrete approximation $u_h \in H_{h}(U) \subset H^{1}_{0}(U)$.
We assume that the complexity for solving for $u_{h}$ is ${\cal O}(
N^{q}_{h})$, where the constant $q \geq 1$ reflects the optimality of
the finite element solver. The cost of solving the approximation of
the influence function $\varphi_{h} \in H_{h}(U)$ is also ${\cal O}(
N^{q}_{h})$. Thus for any ${\bf y}_{s} \in \Gamma_{s}$, the cost for
computing $Q_{h}({\bf y}_{s}):=B({\bf y}_{s};u_{h}({\bf y}_{s}),
\varphi_{h}({\bf y}_{s}))$ is ${\cal O}(N^{q}_{h})$.

Let $\mcS^{m,g}_{\lv}$ be the sparse grid operator characterized
  by $m(i)$ and $g({\bf i})$. Furthermore, let
  $\eta(N_s,m,g,\lv,\Theta_{\beta})$ be the number of the sparse grid
  knots.  The total work for computing the variance $\mathbb{E}[
  (\mcS^{m,g}_{\lv} [Q_{h}({\bf y}_{s})])^2]- $
$\mathbb{E}[\mcS^{m,g}_{\lv} [Q_{h}({\bf y}_{s})]]^{2}$ and the mean
term $\mathbb{E}[\mcS^{m,g}_{\lv} [Q_{h}({\bf y}_{s})]]$ with respect
to a given user tolerance is
\[
W_{Total}(tol) = D_1 N^{q}_{h}(tol)\eta(tol)
\]
for some constant $D_1 > 0$. We now separate the analysis into
three parts:
\begin{enumerate}[(a)]
\item {\bf Truncation:} From the truncation estimate derived in
  section \ref{errorestimates:truncation} we seek $\| Q({\bf y}_s,$
  ${\bf y}_f) - Q({\bf y}_s)\|_{L^{2}_{\rho}(\Gamma)} \leq
  \frac{tol}{3C_{T}}$ with respect to the decay of $\mu_{i}$.  First,
  make the assumption that $B_{\T} = \sup_{x \in U} \sum_{i =
    N_s+1}^{N} \sqrt{\mu_{i}} \|B_{i}(x)\|_{2} \leq C_{D}N^{-l}_s$ for
  some uniformly bounded $C_{D}>0$.  Furthermore, assume that
  $\|b_i(x)\|_{L^{\infty}(U)} \leq D_{D}\sup_{x \in
    U}{\|B_i(x)\|_{2}}$ for $i = 1,\dots N$ where $D_D >0$ is
  uniformly bounded, thus \corg{$C_{\T}=\sup_{x \in U} \sum_{i =
      N_s+1}^{N} \sqrt{\mu_{i}} \|b_{i}(x)\|_{2} $ $\leq$
    $C_{D}D_{D}N^{-l}_s$}.  It follows that $\| Q({\bf y}_s,$ ${\bf
    y}_f) - Q({\bf y}_s)\|_{L^{2}_{\rho}(\Gamma)} \leq
  \frac{tol}{3C_{T}}$ if
\[
B_{\T}\leq C_{D}N^{-l}_s \leq D_{2}tol
\]
for some constant $D_{2} > 0$.  
Finally, we have that
\[
N_{s}(tol) \geq \left\lceil \left( \frac{D_{2} tol}{C_{D}}
\right)^{-1/l} \right\rceil.
\]
\item {\bf Finite Element:} From Section
  \ref{erroranalysis:finiteelement} if
\[
h(tol) \leq \left( \frac{tol}{3C_{FE}
a_{min} \F^{d}_{min} \F^{-2}_{max} C_{\Gamma_s}(r)D_{\Gamma_s}(r) }
\right)^{1/2r}
\]
then $\|Q({\bf y}_{s}) - Q_h({\bf y}_{s})
\|_{L^{2}_{\rho}(\Gamma; H^{1}_{0}(U))} \leq
\frac{tol}{3C_{FE}}$. Assuming that $N_h$ grows as ${\cal O}(h^{-d})$
then
\[
N_h(tol) \geq \left\lceil D_3 \left( \frac{tol}{3C_{FE} a_{min}
  \F^{d}_{min} \F^{-2}_{max} C_{\Gamma_s}(r)D_{\Gamma_s}(r) }
\right)^{-d/2r} \right\rceil
\]
for some constant $D_3 > 0$.
\item {\bf Sparse Grid:} Following the same strategy as in
  \cite{nobile2008a} (equation (3.39)), to simplify the bound
  \eqref{erroranalysis:sparsegrid:estimate} choose $\delta^{*} = (e
  \log{(2)} - 1)/\tilde{C}_2(\sigma)$ and
  $\tilde{C}_2(\sigma)$. Thus
    $\|e\|_{L^{2}_{\hat{\rho}}(\Gamma_{s};H^{1}_{0}(U))} \leq
    \frac{tol}{3C_{SG}C_{T}} \left\| \frac{\rho}{\hat{\rho}}
    \right\|^{-1}_{L^{\infty}(\Gamma_s)} $ if
\[
\eta(tol) \geq \left\lceil \left( \frac{3 \| \rho / \hat{\rho}
\|_{L^{\infty}(\Gamma_s)}
 C_{SG} C_{T} C_{F}F^{N_s}
  \exp(\sigma (\beta,\tilde{\delta}))}{tol}
\right)^{\frac{1 + \log(2N_s)}{\sigma}} \right\rceil
\]
where $C_F = \frac{C_1(\sigma,\delta^{*})}{|1 - C_1(\sigma,\delta^{*})|}$
and $F = \max\{1,C_1(\sigma,\delta^{*})\}$.
\end{enumerate}
\bigskip
Combining (a), (b) and (c) we obtain that for a given user error
tolerance $tol$ the total work is
\[
\begin{split}
W_{Total}(tol) 
&= D_1 N^{q}_{h}(tol,D_3)\eta(\tilde{\delta},\beta,N_s(tol), \| \rho / \hat{\rho}
\|_{L^{\infty}(\Gamma_s)}) \\ 
&= {\cal O}\left( \left( \frac{\| \rho / \hat{\rho}
\|_{L^{\infty}(\Gamma_s)} F^{C tol^{-1/l}}}{tol}
\right)^{\sigma^{-1}(1 + \corr{l^{-1}(\log{2C} - \log{tol})}  )}
 \right). 
\end{split}
\]
for some $C > 0$.

\section{Numerical Results}
\label{numericalresults}
We test our method on a square domain.  Suppose the reference domain
is set $U = (0,1) \times (0,1)$ and stochastically deforms according
to the following rule:
\[
\begin{array}{llll}
  F(x_{1}, x_{2}) = (x_1, x_2) + e(x_1,\omega)(0,\,x_{2}-0.5) & &
  if & x_{2} > 0.5 \\
  F(x_{1}, x_{2}) = (x_{1},\,x_{2}) & & if & 0 \leq x_{2} \leq 0.5
\end{array}
\]
\noindent for some positive constant $c > 0$.  In other words we
deform only the upper half of the domain and fix the \corr{bottom}
half. We set the Dirichlet boundary conditions to zero everywhere
except at the upper border to \corm{$\tilde{u}(x_{1},x_{2})|_{\partial
    U} = g(x_{1})$,} where $g(x_{1}) : =exp \left( \frac{-1}{1 -
  4(x_{1}-0.5)^2} \right)$ (See Figure
\ref{numericalresults:fig1}). This implies that the value at the upper
boundary does not change with boundary perturbation but the solution
does become stochastic with respect to the domain perturbation.
Consider a QoI defined on the bottom half of the reference domain,
which is not deformed, as
\[
Q(u) := \int_{(0,1)} \int_{(0,1/2)}
g(x_{1})g(2x_{2})\corr{\tilde{u}(x_1,x_2,\omega)} \,dx_{1}dx_{2}.
\]
We now show a numerical example with linear decay on the gradient
of the deformation, i.e. the gradient terms
$\sqrt{\mu_{n}} \sup_{x \in U} \|B_n(x)\|$ decay linearly as
$n^{-1}$.

\begin{figure}[h]
\hspace{-5mm}
\psfrag{a}[cb]{\tiny Nominal Domain}
\psfrag{x}[c]{\tiny $x_{1}$} \psfrag{$x_{2}$}[c]{\tiny y} 
\psfrag{quantityofinterest}[][r]{\tiny QoI } 
\psfrag{stochasticdomain}[][r]{\tiny St. Domain}
\psfrag{Realization}[br][br]{\small Realization}
\psfrag{Reference}[br][br]{\small Reference}
\psfrag{d1}[c][bl][1][90]{\tiny \corm{$u|_{\partial U} = 0$}} 
\psfrag{d2}[c][bl][1][0]{\tiny \corm{$u|_{\partial U} = 0$}} 
\psfrag{d3}[c][bl][1][90]{\tiny \corm{$u|_{\partial U} = 0$}} 
\psfrag{d4}[c][bl][1][0]{\tiny \corm{$u|_{\partial U} = 
\corr{exp \left( \frac{-1}{1 - 4(x_{1}-0.5)^2} \right)  }$ } } 
   \includegraphics[ width=5.4in, height=2.2in
  ]{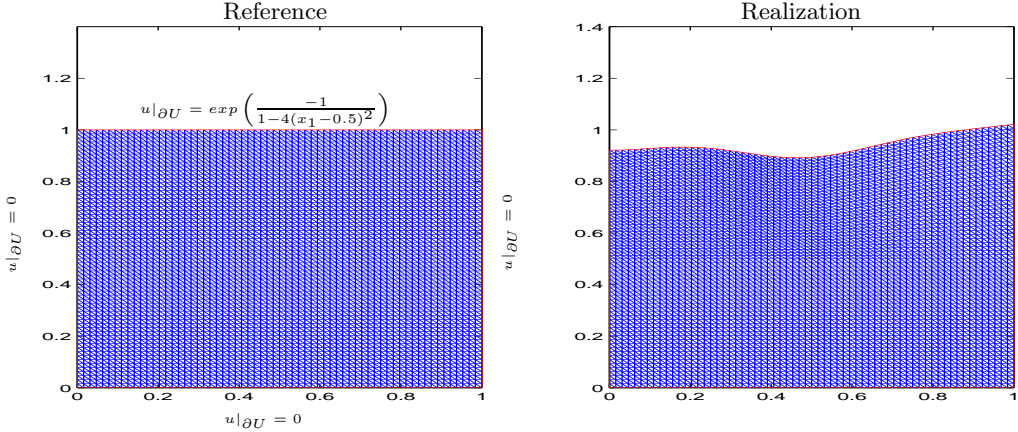} 
  \caption{Stochastic deformation of a square domain. (left) Reference
    square domain with Dirichlet boundary conditions. (right) Vertical
    deformation from stochastic model.}
\label{numericalresults:fig1}
\end{figure}

\begin{figure}[h]
\begin{center}
\begin{tabular}{cc}
\psfrag{N1 = 2aaaaaaaaa}[][c]{\tiny $N_s = 2$}
\psfrag{N1 = 3aaaaaaaaa}[][c]{\tiny $N_s = 3$}
\psfrag{N1 = 4aaaaaaaaa}[][c]{\tiny $N_s = 4$}
\psfrag{N1 = 5aaaaaaaaa}[][c]{\tiny $N_s = 5$}
\psfrag{N1 = 6aaaaaaaaa}[][c]{\tiny $N_s = 6$}
\psfrag{N1 = 7aaaaaaaaa}[][c]{\tiny $N_s = 8$}
\psfrag{knots}{\small knots}
\psfrag{b}[][c]{\small $|\mathbb{E}[Q(u_{ref})] - \mathbb{E}[\mcS^{m,g}_w[Q(u_h)]]|
$}
\psfrag{Mean Error}[][c]{\small Mean Error}
 \includegraphics[width=2.4in,height=2.2in]{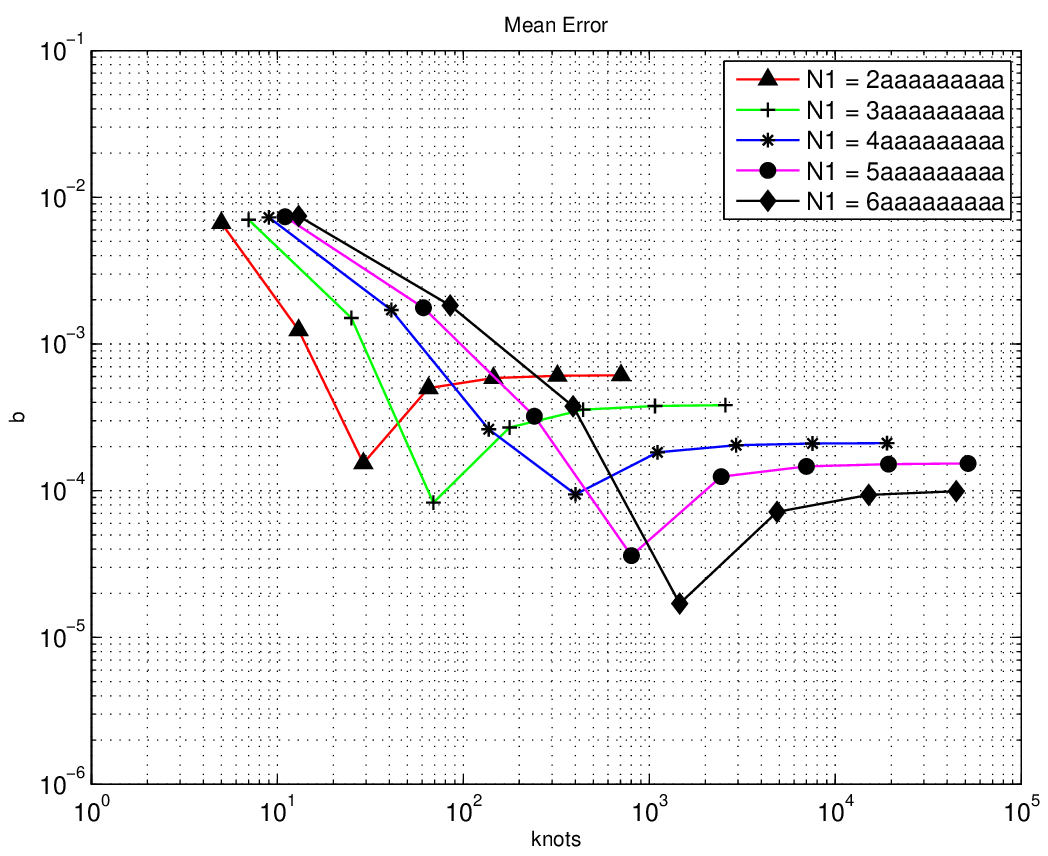} 
&
\psfrag{knots}[][c]{\small knots}
\psfrag{N1 = 2aaaaaaaaa}[][c]{\tiny $N_s = 2$}
\psfrag{N1 = 3aaaaaaaaa}[][c]{\tiny $N_s = 3$}
\psfrag{N1 = 4aaaaaaaaa}[][c]{\tiny $N_s = 4$}
\psfrag{N1 = 5aaaaaaaaa}[][c]{\tiny $N_s = 5$}
\psfrag{N1 = 6aaaaaaaaa}[][c]{\tiny $N_s = 6$}
\psfrag{N1 = 7aaaaaaaaa}[][c]{\tiny $N_s = 8$}
\psfrag{Var Error}[][c]{\small Var Error (Collocation)}
\psfrag{b}[][c]{\small $|Var[Q(u_{ref})] - Var[\mcS^{m,g}_w[Q(u_h)]]|$} 
\includegraphics[width=2.4in,height=2.2in]{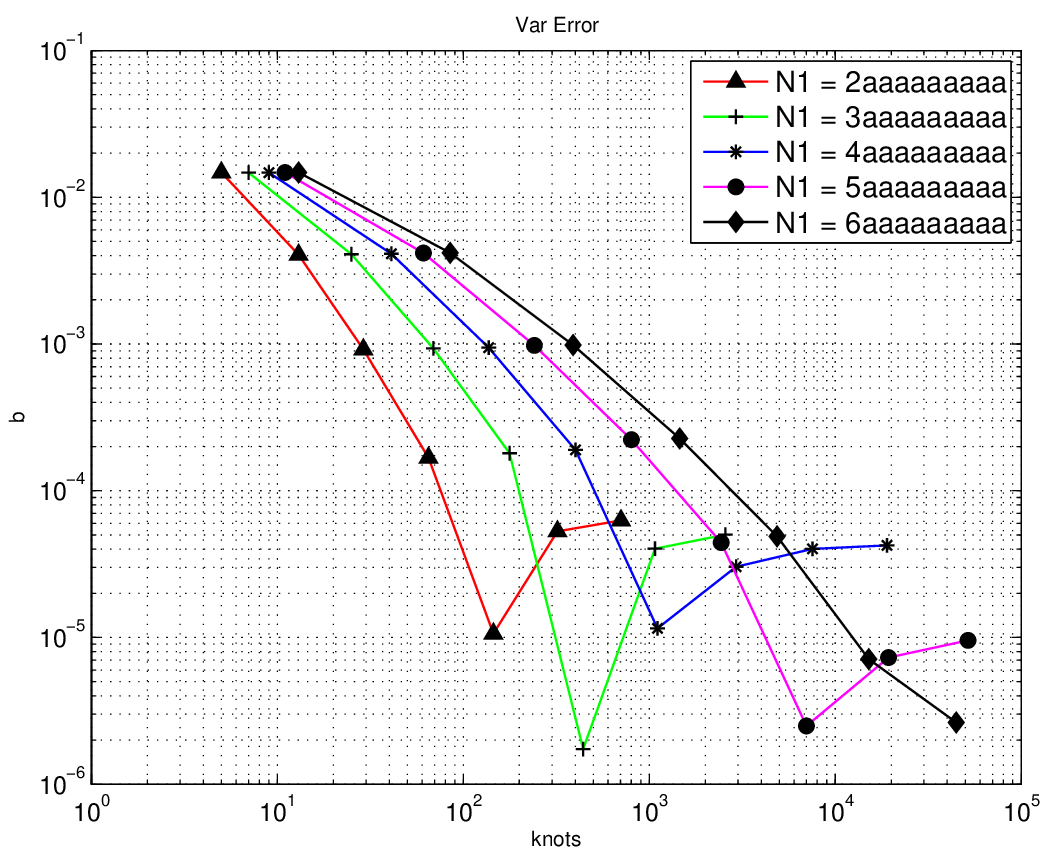} \\
\\
\corr{(a)} & \corr{(b)} \\
\end{tabular}
\end{center}
\caption{Collocation results \corm{(Clenshaw-Curtis multi-linear
    sparse grid)} for $N_s = 2,\dots,6$ with linear decay. (a) Mean
  error with respect to the reference \corm{solution}. Observe that
  the convergence rate decays \corr{subexponentially} until truncation
  saturation is reached.  (b) \corm{Variance error with respect to
    reference solution}. For this case we also observe that the
  convergence rate is faster than polynomial.}
\label{results:fig2}
\end{figure}

\corm{ {\bf Numerical Experiment 1:} For this numerical experiment we compute
  the QoI with a multi-linear Clenshaw-Curtis sparse grid
  \cite{spalg,spdoc}. The parameters of the experiments are as
  follows:}

\begin{enumerate}[(a)]

\item $a(x) = 1$ for all $x \in U$.
\item Stochastic Model. \corr{We split $e(x_1,\omega)$ into large and
  small perturbations as $e(x_{1},\omega): = e_{S}(x_{1},\omega) +
  e_{F}(x_{1},\omega)$, where}
\begin{center}
$e_{S}(x_{1},\omega) := cY_{1}(\omega)  \left(
    \frac{\sqrt{\pi}L}{2} \right)^{1/2} + c \sum_{n = 2}^{N_{s}}
  \sqrt{\mu_{n}}
  \varphi_{n}(x_{1})Y_{n}(\omega); \hspace{1mm}$
  $\corr{e_{F}(x_{1},\omega)} := \corr{c \sum_{n = N_{s}+1}^{N}
  \sqrt{\mu_{n}} \varphi_{n}(x_{1})Y_{n}(\omega)}$.
\end{center}
\item Linear decay $\sqrt{\mu_{n}} :=
  \frac{(\sqrt{\pi}L)^{1/2}}{n}$, $n \in \N$,
\[
\corr{\varphi_{n}(x_{1})} : = \left\{
\begin{array}{cc}
  n^{-1} sin \left( \frac{\lfloor n/2 \rfloor \pi x_{1}}{L_{p}} \right) &
  \mbox{if n is even}\\ n^{-1}cos \left( \frac{\lfloor n/2 \rfloor \pi
      x_{1}}{L_{p}} \right) & \mbox{if n is odd}\\
\end{array}
\right. .
\]
\corr{Thus for $n > 1$ we have that}
\[
\corr{B_n = \left[ \begin{array}{cc}
0 & 0\\
c(x_2 - 0.5) \partial_{x_1} \varphi_n(x_1) & 0 
\end{array}
\right].
}
\]

This implies that $\sup_{x \in U} 
\sigma_{max}(B_{l}(x))$ is bounded by a constant
and we obtain linear decay on the gradient of the deformation.

\item $\{Y_{n}\}_{n = 1}^{N}$ are independent uniform distributed in
  $(-\sqrt{3},\corm{\sqrt{3}).}$

\item \corr{$L = 1/2$}, $L_P = 1$, $c = 0.1533$, $N = 15$.

\item $257 \times 257$ triangular \corm{mesh.}

\item $\mathbb{E}[Q_{h}]$ and \corr{$\var[Q_{h}]$}, are computed with a
  Clenshaw-Curtis isotropic sparse grid (Sparse Grid Toolbox
  V5.1, \cite{spalg,spdoc}).

\item The reference solutions $\var[Q_{h}(u_{ref})]$ and
  $\mathbb{E}[Q_{h}(u_{ref})]$ are computed with \corm{a multi-linear}
  adaptive Sparse Grid ($\approx 30,000$ knots) \cite{Gerstner2003}
  with a $257 \times 257$ mesh for $N = 15$ dimensions.

\item The QoI is normalized by the reference solution $Q(U)$.

\end{enumerate}

In Figure \ref{results:fig2} we show the results of the matlab code
for the \corr{truncated} dimensions $N_s = 2,\dots,6$ and compare the
results with respect to a $N = 15$ dimensional adaptive sparse grid
method collocation with $\approx 30,000$ collocation points
\cite{Gerstner2003}. The computed mean value is 1.0152 and variance is
0.0293 (0.17 std).

In Figure \ref{results:fig2} (a) and (b) the normalized mean and
variance errors are shown for $N_s =2,\dots,6$. For (a) notice the
\corr{subexponential} decay from the sparse grid approximation until
the truncation error and/or finite element error starts to dominate.
\corm{In (b) the variance error decay is actually subexponential
  despite the use of a multi-linear sparse grid, whose performance is
  less than higher order Lagrange polynomial sparse grid shown in
  the second numerical experiment.}


\corm{We now analyze the decay of the truncation error. For $N_s =
  2, \dots, 5$ we compute the mean and variance error as in (g). However,
  for $N_s = 6,\dots,11$ a dimension adaptive sparse grid with 15,000
  to 30,000 sparse grid points is used instead to compute the mean and
  variance. This should be enough to make the error contribution from
  the sparse grid error very small compared to the truncation
  error. The reference solution for the mean and variance is computed
  as in part (h).}

In Figure \ref{results:fig4} we plot the truncation error for (a) the
mean and (b) the variance with respect to the number of dimensions.
We observe that we obtain a convergence rate that appears faster than
the linear decay of the gradient of the stochastic deformation. This
indicates we can further improve the convergence rate of the
truncation estimate.

\begin{figure}
\begin{center}
\begin{tabular}{cc}
\psfrag{Dimension}[c]{\tiny Dimension}
\psfrag{Mean Error}[c]{\tiny Mean Error}
\psfrag{Mean Error vs Dimension}[bl]{\hspace{-7mm} \tiny Mean Truncation Error}
\includegraphics[width=2.4in,height=2.2in]{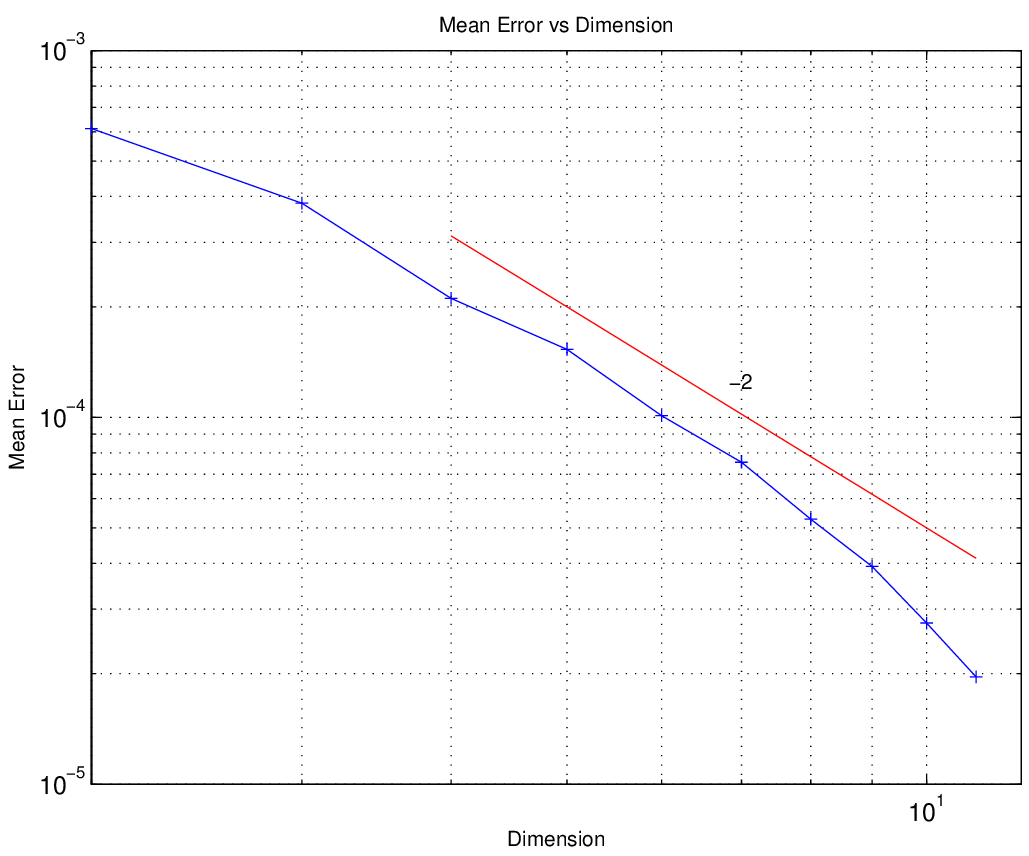} &
\psfrag{Dimension}[c]{\tiny Dimension}
\psfrag{Var Error}[c]{\tiny Var Error}
\psfrag{Var Error vs Dimension}[bl]{\hspace{-9mm} \tiny Variance Truncation Error}
\includegraphics[width=2.4in,height=2.2in]{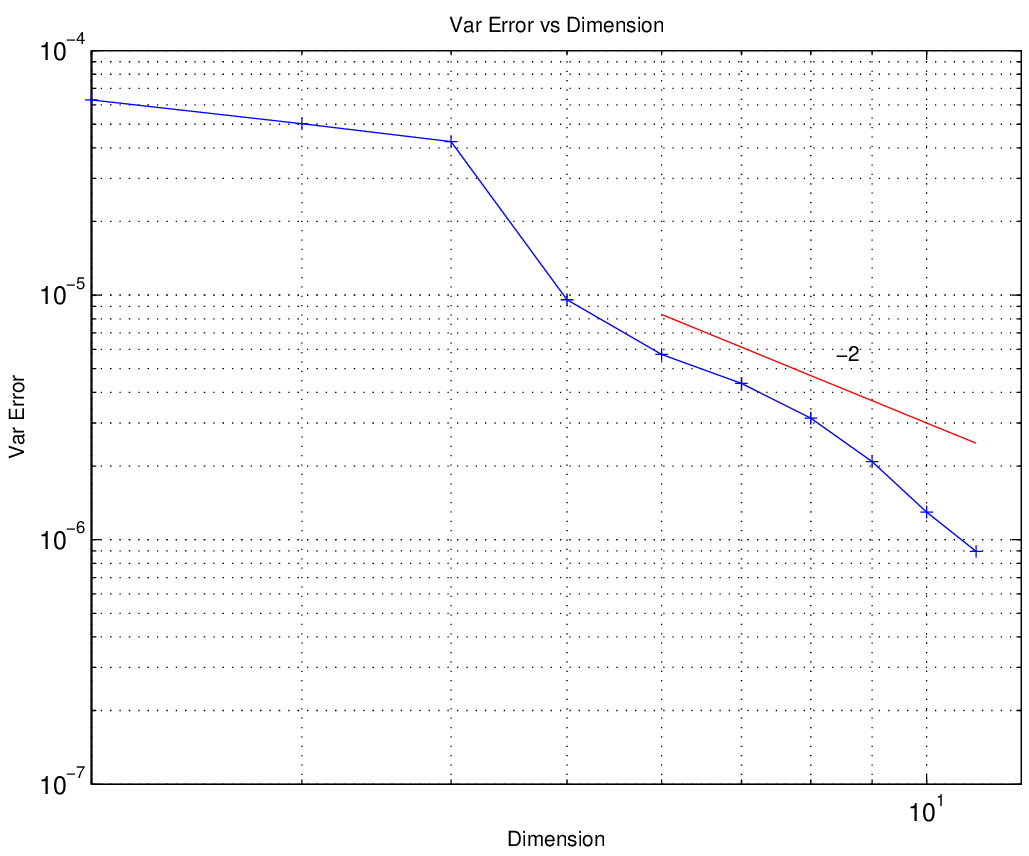} \\
\corr{(a)} & \corr{(b)} \\
\end{tabular}
\end{center}
\caption{Truncation Error with respect to the number of dimensions. (a)
  Mean error. (b) Variance error. In both cases the decay appears
  faster than linear, which is faster than the predicted convergence
  rate.}
\label{results:fig4}
\end{figure}

\bigskip

\corm{ {\bf Numerical Experiment 2:} For this numerical experiment we compute
  the QoI with a higher order Lagrange polynomial sparse grid with
  Chebyshev Gauss-Lobatto abscissas \cite{spalg,spdoc}. For this
  experiment we show the convergence rates for the variance and sparse
  grid error only.  The setup for the numerical experiment is changed
  as follows:}

\begin{enumerate}[(a)]

\item \corm{Since the convergence is much faster than the multi-linear
  sparse grid we increase the mesh size to a $1024 \times 1024$
  triangular mesh.}

\item \corm{$var[Q_h]$ is computed for $N_s = 2,\dots,5$ with the Chebyshev
  Gauss-Lobatto abscissas.}

\item \corm{As a comparison $var[Q_h]$ is also computed for $N_s =
  2,\dots,5$ with the dimension adaptive multi-linear sparse grid with
  15,000 to 20,000 adaptive nodes.}

\end{enumerate}

\begin{figure}[h]
\begin{center}
\psfrag{knots}[][c]{\small knots}
\psfrag{N1 = 2aaaaaaaaa}[][c]{\tiny $N_s = 2$}
\psfrag{N1 = 3aaaaaaaaa}[][c]{\tiny $N_s = 3$}
\psfrag{N1 = 4aaaaaaaaa}[][c]{\tiny $N_s = 4$}
\psfrag{N1 = 5aaaaaaaaa}[][c]{\tiny $N_s = 5$}
\psfrag{N1 = 6aaaaaaaaa}[][c]{\tiny $N_s = 6$}
\psfrag{N1 = 7aaaaaaaaa}[][c]{\tiny $N_s = 8$}
\psfrag{Var Error}[][c]{\small Var Error (Collocation)}
\psfrag{b}[][c]{\small $|Var[Q(u_h)] - Var[\mcS^{m,g}_w[Q(u_h)]]|$} 
\includegraphics[width=3.2in,height=2.8in]{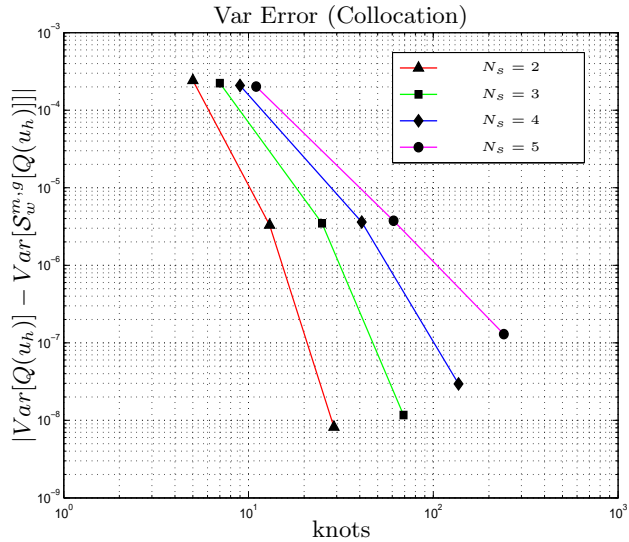} \\
\end{center}
\caption{\corm{Collocation results for higher order Lagrange
    polynomial sparse grid with Chebyshev Gauss-Lobatto  abscissas. Due
    to the higher order Lagrange interpolants the convergence is
    significantly faster compared to the multi-linear sparse grid in
    Figure \ref{results:fig2}}}
\label{results:fig3}
\end{figure}

In Figure \ref{results:fig3} the results for $|Var[Q(u)] -
Var[\mcS^{m,g}_w[Q(u)]]|$ for $N_s = 2, \dots, 5$ are plotted. Due to
the higher order Lagrange interpolants the convergence is
significantly faster compared to the multi-linear sparse grid in
Figure \ref{results:fig2}. Note that we did not do a direct comparison
with $|Var[Q(u_{ref})] - Var[\mcS^{m,g}_w[Q(u_h)]]|$ as the
convergence rate was are so fast that the Truncation error was reached
by the second level of the sparse grid ($w = 2$).

\section{Conclusions}

In this paper we give a rigorous convergence analysis of the
stochastic collocation approach based on isotropic Smolyak grids
  for the approximation of an elliptic PDE defined on a random
  domain. This consists of an analysis of the regularity of the
solution with respect to the parameters describing the domain
perturbation.  Moreover, we derive error estimates both in the
``energy norm'' as well as on functionals of the solution (Quantity of
Interest) for Clenshaw Curtis abscissas that can be easily generalized
to a larger class of sparse grids.

We show that for a linear elliptic partial differential equation with
a random domain the solution can be analytically extended to a
  well defined region $\Theta_{\beta}$ embedded in $\C^{N}$ with
respect to the random variables. This analysis leads to a provable
subexponential convergence rate of the QoI computed with an isotropic
Clenshaw-Curtis sparse grid. We show that the size of this region, and
the rate of convergence, is directly related to the decay of the
  gradient of the stochastic deformation.

As our numerical experiments demonstrate, we are able to solve the
mean and variance of the QoI with moderate deformations of the
domain (leading to a coefficient of variation of the QoI of
  $\approx 0.17$). This is a clear advantage over the perturbation
approaches that are restricted to small deviations. In addition, the
numerical experiments confirm the subexponential rate predicted
from the error estimates.

This approach is well suited for a moderate number of stochastic
  variables but becomes impractical for large problems.  However, we
  can easily extend this approach to anisotropic sparse grids
  \cite{nobile2008b}.

\section*{Acknowledgements} We appreciate the advice from 
Quan Long and Serge Prudhomme. \corm{We are also appreciative of the
  efforts by the reviewers. In particular, one reviewer provided a
  wealth of corrections, comments and advice.}


\begin{thebibliography}{10}
\expandafter\ifx\csname url\endcsname\relax
  \def\url#1{\texttt{#1}}\fi
\expandafter\ifx\csname urlprefix\endcsname\relax\def\urlprefix{URL }\fi
\expandafter\ifx\csname href\endcsname\relax
  \def\href#1#2{#2} \def\path#1{#1}\fi

\bibitem{Zhenhai2005}
Z.~Zhenhai, J.~White, A fast stochastic integral equation solver for modeling
  the rough surface effect computer-aided design, in: IEEE/ACM International
  Conference ICCAD-2005, 2005, pp. 675--682.

\bibitem{Chauviere2006}
J.~S.~H. C.~Chauviere, L.~Lurati., Computational modeling of uncertainty in
  time-domain electromagnetics, SIAM J. Sci. Comput. 28 (2006) 751--775.

\bibitem{Fransos2008}
D.~Fransos, Stochastic numerical methods for wind engineering, Ph.D. thesis,
  Politecnico di Torino (2008).

\bibitem{Tartakovsky2006}
D.~Tartakovsky, D.~Xiu, Stochastic analysis of transport in tubes with rough
  walls, Journal of Computational Physics 217~(1) (2006) 248 -- 259,
  uncertainty Quantification in Simulation Science.

\bibitem{Harbrecht2008}
H.~Harbrecht, R.~Schneider, C.~Schwab, Sparse second moment analysis for
  elliptic problems in stochastic domains, Numerische Mathematik 109 (2008)
  385--414.

\bibitem{Adams1975}
R.~A. Adams, Sobolev Spaces, Academic Press, 1975.

\bibitem{evans1998}
L.~C. Evans, Partial Differential Equations, Vol.~19 of Graduate Studies in
  Mathematics, American Mathematical Society, Providence, Rhode Island, 1998.

\bibitem{Frauenfelder2005}
P.~Frauenfelder, C.~Schwab, R.~A. Todor, Finite elements for elliptic problems
  with stochastic coefficients, Computer Methods in Applied Mechanics and
  Engineering 194~(2-5) (2005) 205 -- 228, selected papers from the 11th
  Conference on The Mathematics of Finite Elements and Applications.

\bibitem{Jabbari}
F.~Jabbari, \href{http://gram.eng.uci.edu/~fjabbari/me270b/me270b.html}{Chapter
  3: Eigenvalues, singular values and pseudo inverse}, University of California
  Irvine, Irvine, California.
\newline\urlprefix\url{http://gram.eng.uci.edu/~fjabbari/me270b/me270b.html}

\bibitem{Ipsen2011}
I.~C.~F. Ipsen, D.~J. Lee, Determinant approximations,
  ArXiv~(arXiv:1105.0437v1).

\bibitem{london1981}
D.~London, A note on matrices with positive definite real part, Proceedings of
  the American Mathematical Society 82~(3) (1981) pp. 322--324.

\bibitem{Krantz1992}
S.~G. Krantz, Function Theory of Several Complex Variables, AMS Chelsea
  Publishing, Providence, Rhode Island, 1992.

\bibitem{Gunning1965}
R.~Gunning, H.~Rossi, Analytic Functions of Several Complex Variables, American
  Mathematical Society, 1965.

\bibitem{Smolyak63}
S.~Smolyak, {Quadrature and interpolation formulas for tensor products of
  certain classes of functions}, Soviet Mathematics, Doklady 4 (1963) 240--243.

\bibitem{Novak_Ritter_00}
V.~Barthelmann, E.~Novak, K.~Ritter, High dimensional polynomial interpolation
  on sparse grids, Advances in Computational Mathematics 12 (2000) 273--288.

\bibitem{Back2011}
J.~B\"{a}ck, F.~Nobile, L.~Tamellini, R.~Tempone, Stochastic spectral galerkin
  and collocation methods for {PDE}s with random coefficients: A numerical
  comparison, in: J.~S. Hesthaven, E.~M. Rønquist (Eds.), Spectral and High
  Order Methods for Partial Differential Equations, Vol.~76 of Lecture Notes in
  Computational Science and Engineering, Springer Berlin Heidelberg, 2011, pp.
  43--62.

\bibitem{nobile2008b}
F.~Nobile, R.~Tempone, C.~Webster,
  An anisotropic sparse
  grid stochastic collocation method for partial differential equations with
  random input data, SIAM Journal on Numerical Analysis 46~(5) (2008)
  2411--2442.
\newline\urlprefix\url{http://epubs.siam.org/doi/pdf/10.1137/070680540}

\bibitem{Chkifa2014}
A.~Chkifa, A.~Cohen, C.~Schwab,
  High-dimensional adaptive
  sparse polynomial interpolation and applications to parametric pdes,
  Foundations of Computational Mathematics 14~(4) (2014) 601--633.
\newblock \href {http://dx.doi.org/10.1007/s10208-013-9154-z}
  {\path{doi:10.1007/s10208-013-9154-z}}.
\newline\urlprefix\url{http://dx.doi.org/10.1007/s10208-013-9154-z}

\bibitem{Babuska2004}
I.~Babuska, R.~Tempone, G.~Zouraris,
  Galerkin
  finite element approximations of stochastic elliptic partial differential
  equations, SIAM Journal on Numerical Analysis 42~(2) (2004) 800--825.
\newline\urlprefix\url{http://epubs.siam.org/doi/pdf/10.1137/S0036142902418680}

\bibitem{Ipsen2008}
I.~Ipsen, R.~Rehman,
  Perturbation bounds
  for determinants and characteristic polynomials, SIAM Journal on Matrix
  Analysis and Applications 30~(2) (2008) 762--776.
\newline\urlprefix\url{http://epubs.siam.org/doi/pdf/10.1137/070704770}

\bibitem{nobile2008a}
F.~Nobile, R.~Tempone, C.~Webster,
  A sparse grid
  stochastic collocation method for partial differential equations with random
  input data, SIAM Journal on Numerical Analysis 46~(5) (2008) 2309--2345.
\newline\urlprefix\url{http://epubs.siam.org/doi/pdf/10.1137/060663660}

\bibitem{spalg}
A.~Klimke, B.~Wohlmuth, Algorithm 847: {spinterp}: Piecewise multilinear
  hierarchical sparse grid interpolation in {MATLAB}, ACM Transactions on
  Mathematical Software 31~(4).

\bibitem{spdoc}
A.~Klimke, {S}parse {G}rid {I}nterpolation {T}oolbox -- user's guide, Tech.
  Rep. {IANS} report 2007/017, University of Stuttgart (2007).

\bibitem{Gerstner2003}
T.~Gerstner, M.~Griebel, Dimension-adaptive tensor-product quadrature,
  Computing 71~(1) (2003) 65--87.

\end{thebibliography}
\end{document}